\documentclass[a4paper,11pt,reqno]{amsart}
\usepackage[left=26mm,right=26mm,top=30mm,bottom=30mm]{geometry}

\usepackage{amssymb}
\usepackage[shortlabels]{enumitem}
\usepackage{bm}
\usepackage{mathtools}
\usepackage{mathrsfs}
\usepackage{array}
\usepackage{xparse}
\usepackage[sort&compress,numbers]{natbib}
\usepackage{caption}
\usepackage{xcolor}

\newtheorem{theorem}{Theorem}[section]
\newtheorem{corollary}{Corollary}[section]
\theoremstyle{plain}

\newtheorem{lemma}{Lemma}[section]
\newtheorem{proposition}{Proposition}[section]

\numberwithin{equation}{section}
\def\<#1>{\mathinner{\langle#1\rangle}}

\makeatletter
\newcommand\makebig[2]{%
 \@xp\newcommand\@xp*\csname#1\endcsname{\bBigg@{#2}}%
 \@xp\newcommand\@xp*\csname#1l\endcsname{\@xp\mathopen\csname#1\endcsname}%
 \@xp\newcommand\@xp*\csname#1r\endcsname{\@xp\mathclose\csname#1\endcsname}%
}
\makeatother

\makebig{biggg} {3.0}
\makebig{Biggg} {3.5}
\makebig{bigggg}{4.0}
\makebig{Bigggg}{4.5}
\makebig{biggggg}{5}
\makebig{Biggggg}{5.5}
\makebig{bigggggg}{6}
\makebig{Bigggggg}{6.5}

\begin{document}
\title{The analysis of vertex feedback stabilisability of a star-shaped network of fluid-conveying pipes}
\author{Xiao Xuan Feng$^{\tt1,\tt3}$}
\author{Gen Qi Xu$^{\tt2}$}
\author{Mahyar Mahinzaeim$^{\tt1,\ast}$}

 \thanks{
\vspace{-1em}\newline\noindent
\newline\noindent
$^{\tt1}$ Research Center for Complex Systems, Aalen University, Germany.
 \newline\noindent
 $^{\tt2}$ Department of Mathematics, Tianjin University, China. 
\newline\noindent
$^{\tt3}$ School of Mathematics and Information Science, Hebei Normal University of Science and Technology, China.
 \newline\noindent
 {\sc Emails}:
 {\tt xiaoxuan.feng@hs-aalen.de},~{\tt gqxu@tju.edu.cn},~{\tt m.mahinzaeim@web.de}.
 \newline\noindent
$^{\ast}$ Corresponding author.
}

\begin{abstract}
It is an outstanding problem whether a pipe-flow system on a star-shaped network is stabilisable by a feedback control on the common vertex. In the present paper we deal with this problem. In particular, we study the equation governing the small vibrations of a stretched elastic pipe conveying fluid in a star-shaped network and examine the question of vertex feedback stabilisability of such a system via control moments. Finding an answer to the question is not straightforward, for the system operator associated with the corresponding closed-loop system is unbounded and nonselfadjoint. An approach to the study of the stabilisation problem for the closed-loop system is presented based on the spectral approach previously introduced by the authors for star graphs of stretched elastic beams. When the tension in the pipes is greater than the square of the fluid-flow velocity, we establish a positive result that in fact gives the strong property of uniform exponential stability of the closed-loop system.
\end{abstract}
\maketitle

\pagestyle{myheadings} \thispagestyle{plain} \markboth{\sc X.\ X.\ Feng, G.\ Q.\ Xu, and M.\ Mahinzaeim}{\sc Star-shaped network of fluid-conveying pipes}

\section{Introduction}\label{sec_01}

In the past two decades one very active area of mathematical systems theory has been the investigation of control, or control-related, problems in mechanics and mathematical physics on networks or metric graphs. Much of the research has focused on the controllability and stabilisability of vibrating elastic -- string, beam, and plate -- networks by insertion of control action into the boundary and connectivity conditions, or, for brevity, simply the \textit{vertex conditions}, and there are already adequate texts on these subjects, e.g.\ \cite{MR4501388,DagerZuazua2006,LagneseEtAl1994,Xu2010graph,MR4397494}. The reason for investigating such systems is easy to explain. For example, remarkable results can be obtained in networked elastic systems when viscous, viscoleastic or thermoelastic damping effects are taken into account or, more generally, when control is exercised. In fact,
one may conjecture that even if an elastic system is not exponentially stable in a nonnetwork setting, it may be possible to stabilise it by vertex feedback control when a network setup is considered. For examples
of such results we refer the reader, e.g., to \cite{MR3891267,MR3799048,MR1814271} and to the above-cited texts. A conclusive answer as to whether or not all vibrating elastic networks can be stabilised by such control action can therefore not be expected, of course, and must be considered individually for each problem. This is specially true when it comes to the stabilisability of flow-induced oscillations in the so-called \textit{pipe-flow systems} on networks, i.e.\ systems describing the vibrations caused by the flow of a fluid through networked pipe systems, and is typical of the particular kind of stabilisation problem we are interested in here.

A resurgence of interest in studying stabilisation problems for pipe-flow systems (or systems related to them) can be seen in the recent papers \cite{AissaEtAl2021,Khemmoudj2021,MahinzaeimEtAl2022,MR4673490,MR4645077,MR4504323,MR4356898}. We discussed in \cite{MahinzaeimEtAl2022} the boundary feedback stabilisation problem for a single elastically vibrating fluid-conveying pipe, assumed long and thin enough that the Euler--Bernoulli beam model is valid, represented by the partial differential equation (with all coefficients nondimensionalised)
\begin{equation}\label{eqa0123}
\frac{\partial^4\bm{\mathfrak{w}}}{\partial s^4}\left(s,t\right)-(\gamma-\eta^2)\,\frac{\partial^2\bm{\mathfrak{w}}}{\partial s^2}\left(s,t\right)+2\beta \eta\frac{\partial^2\bm{\mathfrak{w}}}{\partial s\partial t}\left(s,t\right)+\frac{\partial^2\bm{\mathfrak{w}}}{\partial t^2}\left(s,t\right)=0
\end{equation}
which together with boundary conditions
\begin{align}
\bm{\mathfrak{w}}\left(0,t\right)=\frac{\partial^2 \bm{\mathfrak{w}}}{\partial s^2}\left(0,t\right)&=0,\label{eqa0123a}\\
\frac{\partial^2 \bm{\mathfrak{w}}}{\partial s^2}\left(1,t\right)+\kappa\frac{\partial^2 \bm{\mathfrak{w}}}{\partial s \partial t}\left(1,t\right)&=0,\label{eqa0123b}\\
\frac{\partial^3 \bm{\mathfrak{w}}}{\partial s ^3}\left(1,t\right)-(\gamma-\eta^2)\, \frac{\partial \bm{\mathfrak{w}}}{\partial s }\left(1,t\right)&=0\label{eqa0123c}
\end{align}
and given initial conditions $\bm{\mathfrak{w}}\left(s,0\right)$, $(\partial \bm{\mathfrak{w}}/\partial t)\left(s,0\right)$ forms what we called in \cite{MahinzaeimEtAl2022} a \textit{closed-loop system}, which is just the initial/boundary-value problem for \eqref{eqa0123}. Here $\bm{\mathfrak{w}}\left(s,t\right)$ for $0\le s\le 1$ and $t\in\mathbf{R}_+\left(=\mathbb{R}_+\cup\left\{0\right\}\right)$ represents the transverse deflection or displacement of the pipe, subjected to an externally applied axial force proportional to a parameter $\gamma>0$ (assumed, through boundary condition \eqref{eqa0123c}, to maintain its axial direction; more on this shortly) and carrying the stationary flow of an ideal incompressible fluid with mean velocity $\eta\geq 0$, pinned at $s=0$ (boundary conditions \eqref{eqa0123a}), and subjected at $s=1$ to an angular velocity feedback control proportional to $\kappa\geq 0$ (feedback control moments or damping in boundary condition \eqref{eqa0123b}). It was shown in \cite{MahinzaeimEtAl2022} that unique solutions of the closed-loop system exist, depending continuously on the
initial conditions, and their total energy decreases uniformly exponentially fast as $t\rightarrow\infty$ via the control for $\kappa> 0$, regardless of the values in the pair $\left\{\beta,\eta\right\}$, as long as $\gamma> \eta^2$. (The parameter $\beta$ depends only on the pipe and fluid densities, neither of which is zero, and so by definition $0<\beta<1$.) We note that historically the subject of stability in pipe-flow systems has been most extensively studied in both the mathematical and engineering literature for \eqref{eqa0123} with nondissipative (selfadjoint or
skewadjoint) boundary conditions of standard type, i.e.\ with clamped-free ends (cantilever setting) or with both ends fixed (pinned or clamped), and with or without damping terms in \eqref{eqa0123}. The reader is referred to \cite{PAIDOUSSIS2022103664} for a review of the recent engineering literature. Refer also to \cite[Chapter 3]{Paidoussis2014} or \cite[Section 8.5]{CrandallEtAl1968} for a formal derivation of \eqref{eqa0123}. An explanation of the physical and mathematical meaning of the various standard boundary conditions may be found, e.g., in \cite{HAN1999935}.

A variant of the above initial/boundary-value problem with $\kappa=0$ has been investigated extensively in the literature since its appearance in the early 1950's. In Beck's well-known beam problem (see \cite{Beck1952,Bolotin1963,Ziegler1977}), in which the applied axial force is always tangential to (the displacement curve of) the beam, i.e., the ``follower force'' case, one considers the partial differential equation (taking $\eta=0$ in \eqref{eqa0123})
\begin{equation}\label{eqa0123x}
\frac{\partial^4\bm{\mathfrak{w}}}{\partial s^4}\left(s,t\right)-\gamma\frac{\partial^2\bm{\mathfrak{w}}}{\partial s^2}\left(s,t\right)+\frac{\partial^2\bm{\mathfrak{w}}}{\partial t^2}\left(s,t\right)=0
\end{equation}
with the clamped-free boundary conditions
\begin{equation}
\bm{\mathfrak{w}}\left(0,t\right)=\frac{\partial \bm{\mathfrak{w}}}{\partial s}\left(0,t\right)=\frac{\partial^2 \bm{\mathfrak{w}}}{\partial s^2}\left(1,t\right)=\frac{\partial^3 \bm{\mathfrak{w}}}{\partial s ^3}\left(1,t\right)=0.\label{eqa0123bx}
\end{equation}
%$\bm{\mathfrak{w}}\left(0,t\right)=(\partial \bm{\mathfrak{w}}/\partial s)\left(0,t\right)=0$ and $(\partial^2 \bm{\mathfrak{w}}/\partial s^2)\left(0,t\right)=(\partial^3 \bm{\mathfrak{w}}/\partial s^3)\left(0,t\right)=0$,
The initial/boundary-value problem for \eqref{eqa0123x} with boundary conditions given by \eqref{eqa0123bx} represents a purely \textit{circulatory} system which is not energy-conservative in the classification of Ziegler \cite[Section 1.5]{Ziegler1977}; as is well known this means that the total energy is not necessarily constant for $\gamma>0$ (and in fact also for $\gamma<0$). But as soon as we suppose the applied axial force to act always in a fixed direction along the beam axis it turns out (see \cite[p.\ 91]{Bolotin1963} or \cite[p.\ 112]{Ziegler1977}) that the shear force boundary condition $(\partial^3 \bm{\mathfrak{w}}/\partial s^3)\left(1,t\right)=0$ in \eqref{eqa0123bx} then becomes \eqref{eqa0123c} and energy-conservativeness obtains. This is not surprising, of course, and, properly exploited (as in \cite{MahinzaeimEtAl2022}), it is as a matter of fact exactly what we want as the basis for stabilisability via feedback control.

For networked elastic systems, stabilisability and related problems are often studied for very specific sources of damping in the partial differential equations governing network motion and usually under the assumption that the vertex conditions are nondissipative (e.g.\ see \cite{MR3109894,MR3655800,MR4509852,MR3739755,MR4382292,MR4592978} and the references in the texts cited in the first paragraph above). Yet cases where damping terms are included in the vertex conditions (i.e., the vertex conditions are dissipative) make up the majority of the recent literature on stabilisation problems for networked elastic systems; see the already cited works by Ammari, as well as the recent series of papers by Wang et al.\ \cite{MR4336448,10240878} and Xu et al.\ \cite{MR4027612,MR3908982,MR4027612}.
The case that is the most challenging, however, is when the systems are circulatory and possibly gyroscopic, as is the case for networked pipe-flow systems, and damping is present in the vertex conditions. The most famous classic example of a purely circulatory system in a nonnetwork setting is the problem of Beck, as discussed above. In \cite{MahinzaeimEtAl2021} we have effectively studied \textit{vertex feedback stabilisability} (definition in Section \ref{subsec_01}) for a star-network setup of Beck's problem with damped, pinned-elastic ends. In its essential characteristics -- partial differential equations and vertex conditions with only minor modifications -- this system is similar to an interpretation as a star-network setup of the closed-loop system referred to above when there is no flow, $\eta=0$. Thus this paper in a way is a continuation of \cite{MahinzaeimEtAl2021}. The goal here is to show how our approach in \cite{MahinzaeimEtAl2021} can accomplish the same end of studying vertex feedback stabilisability for a star graph of fluid-conveying pipes. The novelty of the paper, in fact, occurs precisely for vertex feedback stabilisability when $\eta>0$ because it is not obvious how our goal can be achieved without suitable modification of the approach taken in \cite{MahinzaeimEtAl2021}, the main challenge being the spectral properties of the unbounded nonselfadjoint system operator. It should be emphasised that a well-developed ``network stabilisability theory'' has not yet formed for such systems in the literature.

\subsection{System description}

Let us now precisely describe the system we study in this paper. Referring to Fig.\ \ref{fig01} we consider an equilateral $3$-edge metric star graph $\mathbf{G}\coloneqq\left\{\mathbf{V},\mathbf{E}\right\}$, where $\mathbf{V}$ and $\mathbf{E}$ are the sets of vertices $\left\{0\right\}\cup\left\{ a_k\right\}_{k=1}^3$ and edges $\left\{e_k\right\}_{k=1}^3$, respectively, such that each edge $e_k$ connecting the (common) inner vertex at the origin $0$ to the outer vertex $a_k$ is of unit length and is identified with the interval $0\le s_k\le 1$. The value $s_k =1$ corresponds to the outer vertices, and $s_k =0$ corresponds to the inner vertex. We deal here with a nontrivial (nonserial) metric graph containing equal fluid-flow away from the junction of the network.

\begin{figure}[!h]
\begin{center}
\includegraphics[width=0.45\textwidth]{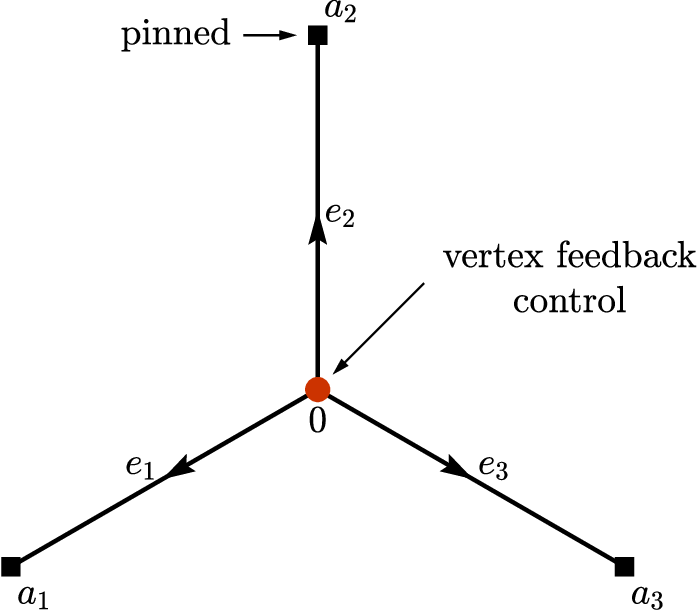}
\end{center}
\caption{A star-shaped network with vertex feedback control.}
\label{fig01}
\end{figure}

Let us henceforth put $s$ in place of the variable $0\le s_k\le 1$, $k=1,2,3$, and let $\bm{\mathfrak{w}}_k\left(s,t\right)$ be the displacement of the $k$-th edge for $0\le s\le 1$ at time $t\in\mathbf{R}_+$ from the equilibrium position of the star graph in the direction orthogonal to it (out-of-plane direction in Fig.\ \ref{fig01}). Consider the partial differential equation
\begin{equation}\label{101}
\frac{\partial^4\bm{\mathfrak{w}}_k}{\partial s ^4}\left(s,t\right)-(\gamma-\eta^2)\,\frac{\partial^2\bm{\mathfrak{w}}_k}{\partial s ^2}\left(s,t\right)+2\beta\eta\frac{\partial^2\bm{\mathfrak{w}}_k}{\partial s \partial t}\left(s,t\right)+\frac{\partial^2\bm{\mathfrak{w}}_k}{\partial t^2}\left(s,t\right) =0,\quad k=1,2,3.
\end{equation}
As initial conditions we require that
\begin{equation}\label{106}
\bm{\mathfrak{w}}_k\left(s, 0\right) = g_k\left (s\right) ,\quad \frac{\partial \bm{\mathfrak{w}}_k}{\partial t}\left(s,0\right)=h_k \left(s\right),\quad k=1,2,3,
\end{equation}
where $g_k$, $h_k$ are known functions having the usual smoothness properties (see Theorem \ref{T-2-1}). For vertex conditions we proceed similarly to \cite{MahinzaeimEtAl2021}. Thus we first assume pinned boundary conditions at the outer vertices $a_k$,
\begin{equation}\label{102}
\bm{\mathfrak{w}}_k\left(1,t\right) =\frac{\partial^2 \bm{\mathfrak{w}}_k}{\partial s^2}\left(1,t\right)=0,\quad k=1,2,3.
\end{equation}
At the inner vertex, the displacement functions $\bm{\mathfrak{w}}_k$ are compatible, so that the continuity condition
\begin{equation}\label{103}
\bm{\mathfrak{w}}_j\left(0,t\right)=\bm{\mathfrak{w}}_k\left(0,t\right),\quad j,k=1,2,3,
\end{equation}
holds. This is the minimal ``geometric'' requirement that has to be satisfied at the inner vertex, regardless of the vertex conditions in moment and shear force, to guarantee that the network does not fall apart. Finally, we list the remaining vertex conditions representing the moments and shear forces. Here we assume that they have to do with vertex feedback control moments applied to a network of beams, i.e.\
\begin{equation}
\frac{\partial^2 \bm{\mathfrak{w}}_k}{\partial s^2}\left(0,t\right)-\alpha\frac{\partial \bm{\mathfrak{w}}_k}{\partial s}\left(0,t\right)-\kappa \frac{\partial^2 \bm{\mathfrak{w}}_k}{\partial s \partial t}\left(0,t\right)=0,\quad k=1,2,3,\label{105}
\end{equation}
with feedback control parameters $\alpha,\kappa\geq 0$, and also have to do with pipe-flow systems on networks which belong to the class of energy-conservative systems, i.e.\
\begin{equation}
\sum_{k=1}^3\left(\frac{\partial^3 \bm{\mathfrak{w}}_k}{\partial s^3}\left(0,t\right)-(\gamma-\eta^2)\,\frac{\partial \bm{\mathfrak{w}}_k}{\partial s}\left(0,t\right)+\beta\eta \frac{\partial \bm{\mathfrak{w}}_k}{\partial t}\left(0,t\right)\right)=0.\label{104}
\end{equation}
The general control problem thus is to choose $\alpha$, $\kappa$, possibly conditioned on the restriction $\gamma> \eta^2$, so as to stabilise, in an appropriate sense, the resulting closed-loop ($\mathfrak{CL}$) system \eqref{101}--\eqref{104}.

Comparison of \eqref{104} with \eqref{eqa0123c} reveals that there is one notable difference namely the addition of extra terms $\beta\eta\,(\partial\bm{\mathfrak{w}}_k/\partial t)\left(0,t\right)$ to the vertex condition, which, as we have just noted, states that the (circulatory, gyroscopic) system is energy conserving when there is no damping in the control, $\kappa=0$. In fact, a justified derivation of \eqref{101} using Hamilton's principle will lead to the vertex conditions \eqref{103}--\eqref{104} when $\kappa=0$. Then \eqref{104} also has the physical interpretation as the force balance condition at the inner vertex just as in \cite[(2.6)]{MahinzaeimEtAl2021}. We digress briefly to illustrate this and at the same time derive other sets of possible vertex conditions which define an energy-conserving star graph of fluid-conveying pipes. This is of some independent interest. (Our discussion of these vertex conditions is brief and intended to be illustrative; the reader interested in greater detail is referred to the papers \cite{MR3349572,MR4072654,MR4433342,MR2070601}.)

Let $\bm{\mathfrak{w}}$ denote the triple of displacement functions $\left\{\bm{\mathfrak{w}}_k\right\}^3_{k=1}$ which, since we are only interested in the formal derivation of the vertex conditions, we may assume sufficiently smooth. We form the Lagrangian functional (the time constants $t_0$, $t_1$ are assumed arbitrary, but fixed, as usual)
\begin{align*}
{L}\left(\bm{\mathfrak{w}}\right)&=\int_{t_0}^{t_1}\frac{1}{2}\left\{\sum_{k=1}^3\int_0^1\left[(1-\beta^2)\left(\frac{\partial \bm{\mathfrak{w}}_k}{\partial t}\right)^2+\beta^2\left(\frac{\partial \bm{\mathfrak{w}}_k}{\partial t}+\frac{\eta}{\beta}\frac{\partial \bm{\mathfrak{w}}_k}{\partial s}\right)^2\right]ds\right\}dt\\
&\qquad-\int_{t_0}^{t_1}\frac{1}{2}\left\{\sum_{k=1}^3\int_0^1\left[\left(\frac{\partial^2 \bm{\mathfrak{w}}_k}{\partial s^2}\right)^2+\gamma\left(\frac{\partial \bm{\mathfrak{w}}_k}{\partial s}\right)^2\right]ds+\alpha\sum_{k=1}^3\left(\frac{\partial \bm{\mathfrak{w}}_k}{\partial s}\left(0,t\right)\right)^2\right\}dt
\end{align*}
wherein $\bm{\mathfrak{w}}_k\left(s,t\right)$ is abbreviated to $\bm{\mathfrak{w}}_k$. The first line is the time integral of the total kinetic energy of the graph, and the second line represents the time integral of the total potential energy, both in nondimensional forms. The first-order variation $\delta{L}\left(\bm{\mathfrak{w}}\right)$ of ${L}\left(\bm{\mathfrak{w}}\right)$ is given by
\begin{align*}
\delta{L}\left(\bm{\mathfrak{w}}\right)&=\int_{t_0}^{t_1}\Biggl\{\sum_{k=1}^3\int_0^1\biggl[\left(\frac{\partial \bm{\mathfrak{w}}_k}{\partial t}+\beta\eta\frac{\partial \bm{\mathfrak{w}}_k}{\partial s}\right)\frac{\partial\delta\bm{\mathfrak{w}}_k}{\partial t}+\left(\beta\eta\frac{\partial \bm{\mathfrak{w}}_k}{\partial t}-(\gamma-\eta^2)\,\frac{\partial \bm{\mathfrak{w}}_k}{\partial s}\right)\frac{\partial\delta\bm{\mathfrak{w}}_k}{\partial s}\\
&\hspace{0.1\linewidth}-\left(\frac{\partial^2 \bm{\mathfrak{w}}_k}{\partial s^2}\right)\frac{\partial^2\delta\bm{\mathfrak{w}}_k}{\partial s^2}\biggr]\,ds-\alpha\sum_{k=1}^3\left(\frac{\partial \bm{\mathfrak{w}}_k}{\partial s}\left(0,t\right)\right)\frac{\partial\delta\bm{\mathfrak{w}}_k}{\partial s}\left(0,t\right)\Biggr\}\,dt.
\end{align*}
By equating $\delta{L}\left(\bm{\mathfrak{w}}\right)$ to zero, it follows from using integration by parts that if the temporal endpoint conditions $\delta\bm{\mathfrak{w}}_k\left(s,t_0\right)=\delta\bm{\mathfrak{w}}_k\left(s,t_1\right)=0$, $k=1,2,3$, are satisfied (as always in Hamilton's principle) then
\begin{align*}
0&=\int_{t_0}^{t_1}\left\{\sum_{k=1}^3\int_0^1\left(\frac{\partial^4\bm{\mathfrak{w}}_k}{\partial s ^4}-(\gamma-\eta^2)\,\frac{\partial^2\bm{\mathfrak{w}}_k}{\partial s ^2}+2\beta\eta\frac{\partial^2\bm{\mathfrak{w}}_k}{\partial s \partial t}+\frac{\partial^2\bm{\mathfrak{w}}_k}{\partial t^2}\right)\delta\bm{\mathfrak{w}}_k\,ds\right\}dt\\
&\qquad-\int_{t_0}^{t_1}\Biggl\{\sum_{k=1}^3\left.\left(\frac{\partial^3 \bm{\mathfrak{w}}_k}{\partial s^3}-(\gamma-\eta^2)\,\frac{\partial \bm{\mathfrak{w}}_k}{\partial s}+\beta\eta\frac{\partial \bm{\mathfrak{w}}_k}{\partial t}\right)\delta\bm{\mathfrak{w}}_k\right|_0^1-\sum_{k=1}^3\left.\left(\frac{\partial^2 \bm{\mathfrak{w}}_k}{\partial s^2}\right)\frac{\partial\delta\bm{\mathfrak{w}}_k}{\partial s}\right|_0^1\\
&\hspace{0.1\linewidth}-\alpha\sum_{k=1}^3\left(\frac{\partial \bm{\mathfrak{w}}_k}{\partial s}\left(0,t\right)\right)\frac{\partial\delta\bm{\mathfrak{w}}_k}{\partial s}\left(0,t\right)\Biggr\}\,dt,
\end{align*}
and hence
\begin{align}
0&=\int_{t_0}^{t_1}\left\{\sum_{k=1}^3\int_0^1\left(\frac{\partial^4\bm{\mathfrak{w}}_k}{\partial s ^4}-(\gamma-\eta^2)\,\frac{\partial^2\bm{\mathfrak{w}}_k}{\partial s ^2}+2\beta\eta\frac{\partial^2\bm{\mathfrak{w}}_k}{\partial s \partial t}+\frac{\partial^2\bm{\mathfrak{w}}_k}{\partial t^2}\right)\delta\bm{\mathfrak{w}}_k\,ds\right\}dt\notag\\
&\qquad+\int_{t_0}^{t_1}\Biggl\{\sum_{k=1}^3\left(\frac{\partial^3 \bm{\mathfrak{w}}_k}{\partial s^3}\left(0,t\right)-(\gamma-\eta^2)\,\frac{\partial \bm{\mathfrak{w}}_k}{\partial s}\left(0,t\right)+\beta\eta\frac{\partial \bm{\mathfrak{w}}_k}{\partial t}\left(0,t\right)\right)\delta\bm{\mathfrak{w}}_k\left(0,t\right)\notag\\
&\hspace{0.1\linewidth}-\sum_{k=1}^3\left(\frac{\partial^2 \bm{\mathfrak{w}}_k}{\partial s^2}\left(0,t\right)-\alpha \frac{\partial \bm{\mathfrak{w}}_k}{\partial s}\left(0,t\right)\right)\frac{\partial\delta\bm{\mathfrak{w}}_k}{\partial s}\left(0,t\right)\Biggr\}\,dt,\label{eq_01axxy}
\end{align}
where we used \eqref{102}. Then, using the conditions
\begin{equation}\label{eq_01bxxy}
\left\{\begin{aligned}
&\bm{\mathfrak{w}}_j\left(0,t\right)=\bm{\mathfrak{w}}_k\left(0,t\right),\quad j,k=1,2,3,\\
&\frac{\partial^2 \bm{\mathfrak{w}}_k}{\partial s^2}\left(0,t\right)-\alpha\frac{\partial \bm{\mathfrak{w}}_k}{\partial s}\left(0,t\right)=0,\quad k=1,2,3,\\
&\sum_{k=1}^3\left(\frac{\partial^3 \bm{\mathfrak{w}}_k}{\partial s^3}\left(0,t\right)-(\gamma-\eta^2)\,\frac{\partial \bm{\mathfrak{w}}_k}{\partial s}\left(0,t\right)+\beta\eta \frac{\partial \bm{\mathfrak{w}}_k}{\partial t}\left(0,t\right)\right)=0
\end{aligned}\right.
\end{equation}
applying at the inner vertex, \eqref{eq_01axxy} becomes equivalent to \eqref{101}, since the $\delta\bm{\mathfrak{w}}_k$ are otherwise completely arbitrary. It is clear that these vertex conditions \eqref{eq_01bxxy} correspond to the same vertex conditions imposed on the $\bm{\mathfrak{w}}_k$ as in \eqref{103}--\eqref{104} for the case where $\kappa=0$.

However there are other sets of vertex conditions arising from \eqref{eq_01axxy}, for example
\begin{equation}\label{eq_01cxxyy}
\left\{\begin{aligned}
&\bm{\mathfrak{w}}_j\left(0,t\right)=\bm{\mathfrak{w}}_k\left(0,t\right),\quad j,k=1,2,3,\\
&\sum_{k=1}^3\frac{\partial \bm{\mathfrak{w}}_k}{\partial s}\left(0,t\right)=0,\\
&\frac{\partial^2 \bm{\mathfrak{w}}_j}{\partial s^2}\left(0,t\right)-\alpha\frac{\partial \bm{\mathfrak{w}}_j}{\partial s}\left(0,t\right)=\frac{\partial^2 \bm{\mathfrak{w}}_k}{\partial s^2}\left(0,t\right)-\alpha\frac{\partial \bm{\mathfrak{w}}_k}{\partial s}\left(0,t\right),\quad j,k=1,2,3,\\
&\sum_{k=1}^3\left(\frac{\partial^3 \bm{\mathfrak{w}}_k}{\partial s^3}\left(0,t\right)-(\gamma-\eta^2)\,\frac{\partial \bm{\mathfrak{w}}_k}{\partial s}\left(0,t\right)+\beta\eta \frac{\partial \bm{\mathfrak{w}}_k}{\partial t}\left(0,t\right)\right)=0,
\end{aligned}\right.
\end{equation}
or
\begin{equation}\label{eq_01cxxy}
\left\{\begin{aligned}
&\bm{\mathfrak{w}}_j\left(0,t\right)=\bm{\mathfrak{w}}_k\left(0,t\right),\quad j,k=1,2,3,\\
&\frac{\partial \bm{\mathfrak{w}}_j}{\partial s}\left(0,t\right)=\frac{\partial \bm{\mathfrak{w}}_k}{\partial s}\left(0,t\right),\quad j,k=1,2,3,\\
&\sum_{k=1}^3\left(\frac{\partial^2 \bm{\mathfrak{w}}_k}{\partial s^2}\left(0,t\right)-\alpha\frac{\partial \bm{\mathfrak{w}}_k}{\partial s}\left(0,t\right)\right)=0,\\
&\sum_{k=1}^3\left(\frac{\partial^3 \bm{\mathfrak{w}}_k}{\partial s^3}\left(0,t\right)-(\gamma-\eta^2)\,\frac{\partial \bm{\mathfrak{w}}_k}{\partial s}\left(0,t\right)+\beta\eta \frac{\partial \bm{\mathfrak{w}}_k}{\partial t}\left(0,t\right)\right)=0,
\end{aligned}\right.
\end{equation}
which also yield an energy-conservative system. These conditions have interesting and physically plausible interpretations too. For example with \eqref{eq_01cxxy} we consider the situation where the pipes are rigidly joined to each other and coupled by a rotational spring with stiffness parameter $\alpha$, and both the tensile forces on the vertex as well as the fluid-flow always remain fixed in direction along the axes of the pipes (which indeed is the case in all three examples considered above, for otherwise the force balance conditions become $\sum_{k=1}^3\,(\partial^3 \bm{\mathfrak{w}}/\partial s^3)\left(0,t\right)=0$, which is consistent with the shear force boundary condition in \eqref{eqa0123bx} and the explanations following it).

The above considerations already indicate that in the closed-loop systems associated with \eqref{101} and either of the vertex conditions \eqref{eq_01bxxy}, \eqref{eq_01cxxyy}, or \eqref{eq_01cxxy}, the system operators will be skewadjoint (cf.\ Lemma \ref{L-2-1}), which implies that their spectra are purely imaginary which in turn implies that the semigroups they generate are certainly not uniformly exponentially stable. The purpose of our work in the sequel is to prove that uniform exponential stability obtains by applying feedback control moments on the inner vertex, i.e.\ for the $\mathfrak{CL}$-system. (Let us note, before proceeding, that Guo and Xie \cite{MR2070993} have studied a somewhat related
%a somewhat related, but much simpler,
stabilisation problem in which a trivial graph consisting of a serial network of two axially moving beams is considered with much simpler vertex conditions; our approach here will therefore, of necessity, be substantially more technical than theirs.)

\subsection{Vertex feedback stabilisability}\label{subsec_01}

In Section \ref{sec_02} we formulate the $\mathfrak{CL}$-system as an abstract system of the form
\begin{equation}\label{exstab012xb}
\dot{{x}}\left(t\right)=\mathcal{T}x\left(t\right), \quad {{x}}\left(0\right)=x_0,
\end{equation} 
on a Hilbert space $\mathscr{X}$ where the system operator $\mathcal{T}:\mathscr{D}\left(\mathcal{T}\right)\subset \mathscr{X}\rightarrow \mathscr{X}$ is the infinitesimal generator of a $C_0$-semigroup in $\mathscr{X}$. In preparation for what follows, we recall that for any $C_0$-semigroup $\mathbb{S}\left(t\right)$ on $\mathscr{X}$, there exist constants $M>0$, $\varpi\in\mathbb{R}$ such that $\left\|\mathbb{S}\left(t\right)\right\|_\mathscr{X}\le Me^{\varpi t}$, $t\in \mathbf{R}_+$, where $\left\|\,\cdot\,\right\|_\mathscr{X}$ is the norm in $\mathscr{X}$. The semigroup is said to be uniformly exponentially stable if the growth rate $\varpi_0<0$ (defined to be the infimum over all $\varpi\in\mathbb{R}$ for which there exists $M\geq 1$ such that the aforementioned inequality holds). So in our case, the $\mathfrak{CL}$-system is then said to be stabilisable via vertex feedback control moments, in short vertex feedback stabilisable, if for some $\varpi_0<0$ there exists a constant $M \ge1 $ such that
\begin{equation}\label{exstab012b}
\left\|{{x}}\left(t\right)\right\|_\mathscr{X}=\left\|\mathbb{S}\left(t\right)x_0\right\|_\mathscr{X}\le Me^{\varpi_0 t}\left\|x_0\right\|_\mathscr{X},\quad t\in\mathbf{R}_+,
\end{equation}
for every solution ${{x}}\left(t\right)$ of \eqref{exstab012xb} corresponding to $x_0\in\mathscr{D}\left(\mathcal{T}\right)$. Then, as $t\rightarrow\infty$, $\left\|{x}\left(t\right)\right\|_\mathscr{X}\rightarrow 0$ uniformly exponentially (where $\left\|{x}\left(t\right)\right\|^2_\mathscr{X}$ is a measure of the energy of the $\mathfrak{CL}$-system at a given time $t$, see Section \ref{sec_02}).

Although there are several approaches to verification of \eqref{exstab012b}, a spectral approach is the most natural one to follow. The spectral approach as it is understood here is variously considered in the literature (see the text by Guo and Wang \cite{GuoWang2019} and the many references therein) under the heading ``Riesz basis approach'' -- as distinguished from the spectral approach based on the well-known Gearhart--Pr\"uss--Huang theorem -- and is used in \cite{MahinzaeimEtAl2021,MahinzaeimEtAl2022} to establish the uniform exponential stability results. In particular we note that, to our knowledge, \cite{MahinzaeimEtAl2022} contains the only rigorous study to date to take the spectral approach to analyse the boundary feedback stabilisability of pipe-flow systems.
%(In fact in the papers \cite{AissaEtAl2021,Khemmoudj2021}, a more direct approach based on Liapunov's second method to the question of boundary feedback stabilisability in a single pipe conveying fluid, which assumes that the controls are achieved via time-varying boundary feedbacks.)

For the spectral approach to work for a semigroup formulation of the $\mathfrak{CL}$-system, in principle, we must show that the system operator $\mathcal{T}$ is a discrete spectral operator in the sense of \cite[Chapters XVIII--XX]{DunfordSchwartz1971}. In this case, the growth rate satisfies $\varpi_0=\sup\left\{\operatorname{Re}\left(\lambda\right)~\middle|~\lambda\in\sigma\left(\mathcal{T}\right)\right\}$, $\sigma\left(\mathcal{T}\right)$ being the spectrum of $\mathcal{T}$, hence the familiar Spectrum Determined Growth Assumption will be satisfied, and consequently one is allowed to use spectral information as a criterion for checking whether $\varpi_0<0$ alone, as it is rigorously justified in \cite{Miloslavskii1985,Roh1982} and elsewhere. We point out that, in general, an operator $\mathcal{Q}$, say, in Hilbert space is a discrete spectral operator if it has a compact resolvent and its root vectors (eigen- and associated vectors) are a Riesz basis for the underlying Hilbert space (see \cite[Corollary XVIII.2.33]{DunfordSchwartz1971}). For this the following property is important: the spectrum $\sigma\left(\mathcal{Q}\right)$ consists of an interpolating sequence $\left\{\lambda_n\right\}$ of eigenvalues, i.e., $\sup\left|\operatorname{Re}\left(\lambda_n\right)\right|<\infty$ and $\inf_{n\neq m}\left|\lambda_n-\lambda_m\right|>0$ for all $n,m\in\mathbb{N}$.
%(and, as a consequence, the eigenvalues are simple).
Note that an operator with these properties automatically has the semigroup generating property (see \cite[Theorem 2.3.5]{CurtainZwart1995}).

The organisation of the paper follows. In Section \ref{sec_02} the operator formalism indicated above is given in an appropriate Hilbert state space so that it will be possible to study well-posedness of the $\mathfrak{CL}$-system in the setting of $C_0$-semigroups. Section \ref{sec_03} deals with a complete spectral analysis (existence, location, multiplicity and asymptotics for eigenvalues) of the system operator. The completeness, minimality, and Riesz basis properties of its root vectors are investigated in Section \ref{sec_04}, where the satisfaction of the Spectrum Determined Growth Assumption is verified. We conclude in Section \ref{sec_05} with a positive result for the vertex feedback stabilisability of the $\mathfrak{CL}$-system.

\section{Operator formulation and well-posedness}\label{sec_02}

We begin this section with a summary of the restrictions on the parameters specified in the Introduction:
\begin{equation*}
\alpha,\kappa\geq 0,\quad 0<\beta<1,\quad \gamma>0,\quad \eta\geq 0,
\end{equation*}
and, as in \cite{MahinzaeimEtAl2022}, it is assumed throughout that
\begin{equation*}
\gamma>\eta^2.
\end{equation*}
Set
\begin{equation*}
\bm{\mathfrak{v}}_k\left(s,t\right)=\frac{\partial \bm{\mathfrak{w}}_k }{\partial t}\left(s,t\right),\quad \bm{{x}}_k\left(s,t\right)=\left(\bm{\mathfrak{w}}_k\left(s,t\right),\bm{\mathfrak{v}}_k\left(s,t\right)\right),\quad k=1,2,3,
\end{equation*}
and
\begin{equation*}
 \bm{{x}}\left(s,t\right)=\left(\bm{{x}}_1\left(s,t\right),\bm{{x}}_2\left(s,t\right),\bm{{x}}_3\left(s,t\right)\right).
\end{equation*}
We then introduce the space $\bm{H}^m_*\left(0,1\right)\coloneqq\left\{w\in \bm{H}^m\left(0,1\right)~\middle|~w\left(1\right)=0\right\}$, $m=1,2$, where  $\bm{H}^m\left(0,1\right)$ denotes the $m$-th order Sobolev space. It follows that $\bm{H}^2_*\left(0,1\right)$, equipped with the inner product
\begin{equation*}
\left(w,\widetilde{w}\right)=\int^1_0w''\left(s\right)\overline{\widetilde{w}''\left(s\right)}\,ds
+(\gamma-\eta^2)\int^1_0 w'\left(s\right)\overline{\widetilde{w}'\left(s\right)}\,ds
+\alpha w'\left(0\right)\overline{\widetilde{w}'\left(0\right)},
\end{equation*}
is a Hilbert space.

Let now $\bm{L}_2\left(\mathbf{G}\right)$ be the metric space of vector-valued functions $v=\left(v_1,v_2,v_3\right)$ for which $v_k\in \bm{L}_2\left(0,1\right)$, $k=1,2,3$. We similarly define the space $\bm{H}^2_*\left(\mathbf{G}\right)$ of vector-valued functions $w=\left(w_1,w_2,w_3\right)$ for which $w_k\in \bm{H}^2_*\left(0,1\right)$, $k=1,2,3$, and $w_j\left(0\right)=w_k\left(0\right)$, $j,k=1,2,3$. In the Hilbert state space $\mathscr{X}=\bm{H}^2_*\left(\mathbf{G}\right)\times \bm{L}_2\left(\mathbf{G}\right)$, i.e.,
\begin{equation*}
\mathscr{X} \coloneqq\left\{
x=\left\{x_k\right\}^3_{k=1}~\middle|~\begin{gathered}
x_k=\left(w_k,v_k\right)\in \bm{H}^2_*\left(0,1\right)\times \bm{L}_2\left(0,1\right),\\
w_j\left(0\right)=w_k\left(0\right),\quad j,k=1,2,3
\end{gathered}
\right\}
\end{equation*}
with the inner product (inducing an energy-motivated norm in $\mathscr{X}$)
\begin{gather*}
\left(x,\widetilde{x}\right)_{\mathscr{X}} \coloneqq\left(w,\widetilde{w}\right)_{\bm{H}^2_*\left(\mathbf{G}\right)}+\left(v,\widetilde{v}\right)_{\bm{L}_2\left(\mathbf{G}\right)},
\end{gather*}
where
\begin{gather*}
\left(w,\widetilde{w}\right)_{\bm{H}^2_*\left(\mathbf{G}\right)}=\sum_{k=1}^3\left[\int^1_0w_k''\left(s\right)\overline{\widetilde{w}_k''\left(s\right)}\,ds
+(\gamma-\eta^2)\int^1_0 w_k'\left(s\right)\overline{\widetilde{w}_k'\left(s\right)}\,ds
+\alpha w_k'\left(0\right)\overline{\widetilde{w}_k'\left(0\right)}\right],\\
\left(v,\widetilde{v}\right)_{\bm{L}_2\left(\mathbf{G}\right)}=\sum_{k=1}^3\int^1_0v_k\left(s\right)\overline{\widetilde{v}_k\left(s\right)}\,ds,
\end{gather*}
we define the operators $\mathcal{A}$, $\mathcal{B}$ on the domains
\begin{align}
\mathscr{D}\left(\mathcal{A}\right)&=\left\{x=\left\{x_k\right\}^3_{k=1}\in \mathscr{X}~\middle|
~\begin{gathered}
x_k=\left(w_k,v_k\right)\in (\bm{H}^4\left(0,1\right)\cap \bm{H}^2_*\left(0,1\right))\times \bm{H}^2_*\left(0,1\right), \\
w''_k\left(1\right)=0,\quad w''_k\left(0\right)-\alpha w'_k\left(0\right)-\kappa v'_k\left(0\right)=0,\quad k=1,2,3,\\
\sum_{k=1}^3\,\bigl[w^{\left(3\right)}_k\left(0\right)-(\gamma-\eta^2)\, w_{k}'\left(0\right)+\beta\eta v_k\left(0\right)\bigr]=0
\end{gathered}\right\},\label{204}\\
\mathscr{D}\left(\mathcal{B}\right)&=\left\{x=\left\{x_k\right\}^3_{k=1}\in \mathscr{X}~\middle|
~x_k=\left(w_k,v_k\right)\in \bm{H}^2_*\left(0,1\right)\times \bm{H}^1_*\left(0,1\right),\quad k=1,2,3\right\}\label{206}
\end{align}
by
\begin{align}
\mathcal{A}x&\coloneqq\bigl\{\bigl(v_k,-w_k^{\left(4\right)}+(\gamma-\eta^2)\,w''_k\bigr)\bigr\}_{k=1}^3,\label{203}\\
\mathcal{B}x&\coloneqq\bigl\{\bigl(0,-2\beta\eta v_k'\bigr)\bigr\}_{k=1}^3,\label{205}
\end{align}
respectively. We recall that the continuity condition $w_j\left(0\right)=w_k\left(0\right)$, $j,k=1,2,3$, is ``hidden'' in the definition of $\mathscr{X}$. Clearly $\mathscr{D}\left(\mathcal{A}\right)\subset \mathscr{D}\left(\mathcal{B}\right)$, and one may verify quite readily that $\mathcal{B}$ is relatively compact with respect to $\mathcal{A}$ (in the sense of \cite[Section IV.1.3]{Kato1995}).

Define $\left[x\left(t\right)\right]\left(s\right)\coloneqq \bm{{x}}\left(s,t\right)$. The $\mathfrak{CL}$-system may be formulated abstractly thus:
\begin{equation}\label{207}
\left\{\begin{split}
\dot{{x}}\left(t\right)&=\mathcal{T}{x}\left(t\right),\quad \mathcal{T}\coloneqq \mathcal{A}+\mathcal{B},\quad \mathscr{D}\left(\mathcal{T}\right)=\mathscr{D}\left(\mathcal{A}\right),\\
{{x}}\left(0\right)&=x_0,
\end{split}\right.
\end{equation}
where
\begin{equation*}
{x}\left(t\right)=\left\{\left(\bm{\mathfrak{w}}_k\left(\,\cdot\,,t\right),\bm{\mathfrak{v}}_k\left(\,\cdot\,,t\right)\right)\right\}^3_{k=1},\quad x_0=\left\{\left(g_k,h_k\right)\right\}^3_{k=1}.
\end{equation*}
Substitute ${x}\left(t\right)=x\exp\left(\lambda t\right)$, $x\in\mathscr{X}$, in \eqref{207} and note for later reference that
\begin{equation}\label{eq_01xaa02}
\mathcal{T}x=\lambda x,\quad x\in\mathscr{D}\left(\mathcal{A}\right),\quad \lambda\in\mathbb{C},
\end{equation}
which is the spectral problem for the $\mathfrak{CL}$-system with spectral parameter $\lambda$.

Our central well-posedness result is the following:
\begin{theorem}\label{T-2-1}
The $\mathfrak{CL}$-system is well posed in the sense that \eqref{207} has for any $x_0\in \mathscr{D}\left(\mathcal{A}\right)$ a unique solution ${x}\in \bm{C}^1\left(\left(0,\infty\right); \mathscr{X}\right)\cap \bm{C}\left(\mathbf{R}_+; \mathscr{D}\left(\mathcal{A}\right)\right)$ given by
\begin{equation*}
{x}\left(t\right)=\mathbb{S}\left(t\right)x_0
\end{equation*}
where $\mathbb{S}\left(t\right)$ is a contractive $C_0$-semigroup on $\mathscr{X}$ with infinitesimal generator $\mathcal{T}$.
\end{theorem}

We will prove the theorem with the help of the following lemma.
\begin{lemma}\label{L-2-1}
The following statements hold:
\begin{enumerate}[\normalfont(1)]
\item\label{T-3-2-a} $\mathcal{T}$ is closed and densely defined.
\item\label{T-3-2-b} $0\in\varrho\left(\mathcal{T}\right)$ (the resolvent set of $\mathcal{T}$) and $\mathcal{T}^{-1}$ is compact.
\item\label{T-3-2-c} $\mathcal{T}$ is maximal dissipative for $\kappa>0$ and skewadjoint for $\kappa=0$.
\end{enumerate}
\end{lemma}
\begin{proof}
We begin with a proof of statement \ref{T-3-2-b} and show, first of all, that $\mathcal{T}$ is injective and surjective. To this end for $\lambda=0$ we consider the solution $x$ of \eqref{eq_01xaa02}. Then we clearly have $v_k=0$, $k=1,2,3$, and the $w_k=w_k\left(\lambda,s\right)$ satisfy the boundary-eigenvalue problem
\begin{equation}\label{208}
\left\{\begin{aligned}
&w^{\left(4\right)}_k-(\gamma-\eta^2)\, w''_k=0,\quad k=1,2,3,\\
&w_k\left(1\right)=w''_k\left(1\right)=0,\quad k=1,2,3,\\
&w_j\left(0\right)=w_k\left(0\right),\quad j,k=1,2,3,\\
&w''_k\left(0\right)-\alpha w'_k\left(0\right)=0,\quad k=1,2,3,\\
&\sum^3_{k=1}\,\bigl[w^{\left(3\right)}_k\left(0\right)-(\gamma-\eta^2)\, w'_k\left(0\right)\bigr]=0.
\end{aligned}\right.
\end{equation}
We multiply the differential equations in \eqref{208} by the conjugate of $w_k$, $k=1,2,3$, and integrate from $0$ to $1$. Integrating by parts, making use of the vertex conditions $w_k\left(1\right)=w''_k\left(1\right)=0$, $w''_k\left(0\right)-\alpha w'_k\left(0\right)=0$, $k=1,2,3$, we get
\begin{align*}
0&=\int_{0}^{1}\left|w_k''\left(s\right)\right|^2ds+(\gamma-\eta^2)\int_{0}^{1}\left|w_k'\left(s\right)\right|^2ds+\alpha\left|w_k'\left(0\right)\right|^2\\
&\qquad-\bigl[w^{\left(3\right)}_k\left(0\right)-(\gamma-\eta^2)\, w'_k\left(0\right)\bigr]\,\overline{w_k\left(0\right)},\quad k=1,2,3.
\end{align*}
Subsequent summation over $k=1,2,3$, using the continuity condition $w_j\left(0\right)=w_k\left(0\right)$, $j,k=1,2,3$, and force balance condition $\sum^3_{k=1}\,\bigl[w^{\left(3\right)}_k\left(0\right)-(\gamma-\eta^2)\, w'_k\left(0\right)\bigr]=0$, yields
\begin{equation*}
0=\sum_{k=1}^3\left[\int^1_0\left|w_k''\left(s\right)\right|^2ds
+(\gamma-\eta^2)\int^1_0 \left|w_k'\left(s\right)\right|^2ds +\alpha \left|w_k'\left(0\right)\right|^2\right]=\left\|w\right\|^2_{\bm{H}_*^2\left(\mathbf{G}\right)}.
\end{equation*}
The fact that $v_k=0$, $k=1,2,3$, implies that $\left\|v\right\|_{\bm{L}_2\left(\mathbf{G}\right)}=0$ and hence that $\left\|x\right\|_{\mathscr{X}}=0$, or equivalently, that $x=0$. It follows from this that $\operatorname{Ker}\mathcal{T}=0$ and so $\mathcal{T}$ is injective. Thus, $0$ is not an eigenvalue.

For the surjectivity, we proceed as follows. Let $\widetilde{x}\in \mathscr{X}$, $x\in \mathscr{D}\left(\mathcal{A}\right)$ (arbitrary) and consider the equation
\begin{equation}\label{eq12xa34}
\mathcal{T}x=\widetilde{x},
\end{equation}
or equivalently in coordinates,
\begin{equation}\label{213}
\left\{\begin{aligned}
&v_k=\widetilde{w}_k,\quad -w^{\left(4\right)}_k+(\gamma-\eta^2)\, w''_k-2\beta \eta v_{k}'=\widetilde{v}_k,\quad k=1,2,3,\\
&w_k\left(1\right)=w''_k\left(1\right)=0,\quad k=1,2,3,\\
&w_j\left(0\right)=w_k\left(0\right),\quad j,k=1,2,3,\\
&w''_k\left(0\right)-\alpha w'_k\left(0\right)-\kappa v'_k\left(0\right)=0,\quad k=1,2,3,\\
&\sum^3_{k=1}\,\bigl[w^{\left(3\right)}_k\left(0\right)-(\gamma-\eta^2)\, w_{k}'\left(0\right)+\beta\eta v_k\left(0\right)\bigr]=0.
\end{aligned}\right.
\end{equation}
Write the differential equations in \eqref{213} in weak form as
\begin{equation*}
\int_0^1\biggl[-w^{\left(4\right)}_k\left(s\right)+(\gamma-\eta^2)\, w''_k\left(s\right)-2\beta \eta v_{k}'\left(s\right)\biggr]\,\overline{\phi_k\left(s\right)}\,ds=\int_0^1\widetilde{v}_k\left(s\right)\overline{\phi_k\left(s\right)}\,ds,\quad k=1,2,3,
\end{equation*}
for the $\phi_k$ in an appropriate class of test functions satisfying the conditions $\phi_k\left(1\right)=0$, $k=1,2,3$, and $\phi_j\left(0\right)=\phi_k\left(0\right)$, $j,k=1,2,3$. Again,
\begin{align*}
&\sum_{k=1}^{3}\int_{0}^{1}\left(\widetilde{v}_k\left(s\right)+2\beta\eta\widetilde{w}_k'\left(s\right)\right)\overline{\phi_k\left(s\right)}\,ds+\beta\eta\sum_{k=1}^{3} \widetilde{w}_k\left(0\right)\overline{\phi_k\left(0\right)}+\kappa\sum_{k=1}^{3} \widetilde{w}_k'\left(0\right)\overline{\phi_k'\left(0\right)}\\
&\hspace{0.1\linewidth}=-\sum_{k=1}^{3}\left[\int_{0}^{1}w_k''\left(s\right)\overline{\phi_k''\left(s\right)}\,ds+(\gamma-\eta^2)\int_{0}^{1}w_k'\left(s\right)\overline{\phi_k'\left(s\right)}\,ds+\alpha w_k'\left(0\right)\overline{\phi_k'\left(0\right)}\right]
\end{align*}
employing integration by parts and using the vertex conditions, and where we have taken into account that $v_k=\widetilde{w}_k$, $k=1,2,3$. We define the bilinear form $\left<\,\cdot\,,\,\cdot\,\right>$ on $\bm{H}_*^2\left(\mathbf{G}\right)$ by
\begin{equation*}
\left<u,\widetilde{u}\right>\coloneqq\sum_{k=1}^{3}\left[\int_{0}^{1}u_k''\left(s\right)\overline{\widetilde{u}_{k}''\left(s\right)}\,ds+(\gamma-\eta^2)\int_{0}^{1}u_k'\left(s\right)\overline{\widetilde{u}_{k}'\left(s\right)}\,ds+\alpha u_k'\left(0\right)\overline{\widetilde{u}_{k}'\left(0\right)}\right]
\end{equation*}
for any $u,\widetilde{u}\in \bm{H}_*^2\left(\mathbf{G}\right)$. Clearly from Schwarz's inequality,
\begin{equation*}
\left|\left<u,\widetilde{u}\right>\right|\leq\left\|u\right\|_{\bm{H}_*^2\left(\mathbf{G}\right)}\left\|\widetilde{u}\right\|_{\bm{H}_*^2\left(\mathbf{G}\right)},\quad u,\widetilde{u}\in \bm{H}_*^2\left(\mathbf{G}\right).
\end{equation*}
Moreover, $\left<u,u\right>=\left\|u\right\|^2_{\bm{H}_*^2\left(\mathbf{G}\right)}$, and so $\left<\,\cdot\,,\,\cdot\,\right>$ is coercive. Now define
\begin{equation*}
F\left(\phi\right)\coloneqq\sum_{k=1}^{3}\int_{0}^{1}\left(\widetilde{v}_k\left(s\right)+2\beta\eta\widetilde{w}_k'\left(s\right)\right)\overline{\phi_k\left(s\right)}\,ds+\beta\eta\sum_{k=1}^{3} \widetilde{w}_k\left(0\right)\overline{\phi_k\left(0\right)}+\kappa\sum_{k=1}^{3} \widetilde{w}_k'\left(0\right)\overline{\phi_k'\left(0\right)},
\end{equation*}
a bounded conjugate linear functional of $\phi$ in $\bm{H}_*^2\left(\mathbf{G}\right)$ for fixed $\widetilde{x}=\left\{\left(\widetilde{w}_k,\widetilde{v}_k\right)\right\}^3_{k=1}$. Then there exists a unique $w=\left\{w_k\right\}^3_{k=1}\in \bm{H}_*^2\left(\mathbf{G}\right)$ such that $\left<w,\phi\right>+F\left(\phi\right)=0$ (by the Lax--Milgram theorem) and therefore that $w_k=w_k\left(s\right)$, $k=1,2,3$, satisfies the differential equations in \eqref{213}. Since $x=\left\{\left(w_k,v_k\right)\right\}^3_{k=1}=\left\{\left(w_k,\widetilde{w}_k\right)\right\}^3_{k=1}$, we have $\left\{\left(w_k,\widetilde{w}_k\right)\right\}^3_{k=1}\in \mathscr{D}\left(\mathcal{A}\right)$ and \eqref{eq12xa34} is satisfied for any given $\widetilde{x}=\left\{\left(\widetilde{w}_k,\widetilde{v}_k\right)\right\}^3_{k=1}\in \mathscr{X}$. This proves surjectivity, and hence bijectivity of $\mathcal{T}$. Anticipating the result of statement \ref{T-3-2-c} that $\mathcal{T}$ is maximal dissipative (and hence closed) we have by the closed graph theorem that the inverse $\mathcal{T}^{-1}$ of $\mathcal{T}$ is closed and bounded, so it remains only to show that $\mathcal{T}^{-1}$ is compact. But this is clear, for $\mathscr{D}\left(\mathcal{A}\right)\subset (\bm{H}^4\left(\mathbf{G}\right)\cap \bm{H}_*^2\left(\mathbf{G}\right))\times \bm{H}_*^2\left(\mathbf{G}\right)$ and the Sobolev embeddings $ \bm{H}^4\left(\mathbf{G}\right)\cap \bm{H}_*^2\left(\mathbf{G}\right)\xhookrightarrow{} \bm{H}_*^2\left(\mathbf{G}\right) \xhookrightarrow{} \bm{L}_2\left(\mathbf{G}\right)$ are compact.

In order to establish statement \ref{T-3-2-c}, we compute for any $x\in\mathscr{D}\left(\mathcal{A}\right)$,
\begin{equation}\label{214}
2\operatorname{Re}\left(\mathcal{T}x,x\right)_{\mathscr{X}}=\left(\mathcal{T}x,x\right)_{\mathscr{X}}+\left(x,\mathcal{T}x\right)_{\mathscr{X}} =-2\kappa \sum_{k=1}^3\left|v_k'\left(0\right)\right|^2,
\end{equation}
after the usual integrations by parts and using the vertex conditions. So $\mathcal{T}$ is dissipative for $\kappa>0$. We have, in fact, shown that $\mathcal{T}$ is maximal dissipative. That is, since $0\in\varrho\left(\mathcal{T}\right)$ and the resolvent set is open, we may find $\lambda>0$ in $\varrho\left(\mathcal{T}\right)$ such that the range $\operatorname{Im}\left(\lambda I-\mathcal{T}\right)=\mathbb{X}$; hence the maximality of $\mathcal{T}$. It is skewsymmetric when $\kappa=0$ for then we have $\left(\mathcal{T}x,x\right)_{\mathscr{X}}+\left(x,\mathcal{T}x\right)_{\mathscr{X}} =0$ from \eqref{214}. The resulting skewadjointness of $\mathcal{T}$ for $\kappa=0$ can be obtained for example by the same standard arguments as in the proof of \cite[Lemma 3.2]{MahinzaeimEtAl2021}.

Finally, statements \ref{T-3-2-b} and \ref{T-3-2-c} imply statement \ref{T-3-2-a}.
\end{proof}

\begin{proof}[Proof of Theorem \ref{T-2-1}]
Since, by Lemma \ref{L-2-1}, $\mathcal{T}$ is a closed, densely defined, maximal dissipative operator in $\mathscr{X}$, the proof follows immediately from the Lumer--Phillips theorem, see \cite[Theorem 1.4.3]{Pazy1983}.
\end{proof}

\section{Spectral analysis}\label{sec_03}

We return to the spectral problem \eqref{eq_01xaa02} and analyse the spectrum of the system operator $\mathcal{T}$ in detail. The reader is referred to \cite[Definitions 4.1--4.3]{MahinzaeimEtAl2021} or to standard textbooks for the standard definitions from functional analysis of the spectral theory of linear operators in Hilbert space.

It follows from Lemma \ref{L-2-1} that $\mathcal{T}$ has compact resolvent and $0\in\varrho\left(\mathcal{T}\right)$. It is known then (see, e.g., \cite[Corollary XI.8.4]{GohbergEtAl1990}) that the spectrum $\sigma\left(\mathcal{T}\right)$ is a purely discrete set, consisting only of isolated eigenvalues of finite algebraic multiplicity (normal eigenvalues) which accumulate only at infinity. Moreover, because $\mathcal{T}$ is maximal dissipative, for any $\lambda\in\sigma\left(\mathcal{T}\right)$ we have $\operatorname{Re}\left(\lambda\right)\leq 0$. Further information about the location of the eigenvalues are obtained in the next result.
\begin{theorem}\label{T-3-1}
The spectrum of $\mathcal{T}\coloneqq \mathcal{A}+\mathcal{B}$ as defined by \eqref{204}--\eqref{205} is symmetric with respect to the real axis of the complex plane, the eigenvalues of $\mathcal{T}$ being confined to the open left half-plane when $\kappa>0$.
\end{theorem}
\begin{proof}
Since $\mathcal{T}$ is a real operator we obtain the first assertion about symmetry of the spectrum. Indeed, conjugation of \eqref{eq_01xaa02} shows that $\overline{x}$ satisfies its conjugate spectral problem and is the eigenvector of $\mathcal{T}$ which corresponds to the eigenvalue $\overline{\lambda}$. To prove the second assertion, we take the inner product of \eqref{eq_01xaa02} with the corresponding $x$ to obtain for the real part of the resulting expression, taking into account \eqref{214},
\begin{equation*}
\operatorname{Re}\left(\lambda\right)=\frac{\operatorname{Re}\left(\mathcal{T}x ,x \right)_{\mathscr{X}} }{\left\|x \right\|_{\mathscr{X}}^2}\leq 0.
\end{equation*}
We must show that if $\kappa>0$ then $\operatorname{Re}\left(\lambda\right)<0$. Let $\lambda$ be an eigenvalue with $\operatorname{Re}\left(\lambda\right)=0$ and let $x= \left\{\left(w_k, v_k\right)\right\}_{k=1}^3$ be the corresponding eigenvector. Then, because $v_k = \lambda w_k$, $k=1,2,3$, replacing the $v_k'\left(0\right)$ in \eqref{214} by $\lambda w_k'\left(0\right)$ we arrive at
\begin{equation*}
\operatorname{Re}\left(\mathcal{T}x,x\right)_{\mathscr{X}}= -\kappa\left|\lambda\right|^2 \sum_{k=1}^3\left|w_k'\left(0\right)\right|^2=0
\end{equation*}
and
\begin{equation*}
\sum_{k=1}^3\left|w_k'\left(0\right)\right|^2=0
\end{equation*}
since $\kappa>0$, $\lambda\ne 0$. So $w_k'\left(0\right)=0$, $k=1,2,3$, and the $w_k=w_k\left(\lambda,s\right)$ satisfy the boundary-eigenvalue problem
\begin{equation}\label{302}
\left\{\begin{aligned}
&w^{\left(4\right)}_k-(\gamma-\eta^2)\, w''_k+2\lambda\beta\eta w'_k=-\lambda^2 w_k,\quad k=1,2,3,\\
&w_k\left(1\right)=w''_k\left(1\right)=0,\quad k=1,2,3,\\
&w_j\left(0\right)=w_k\left(0\right),\quad j,k=1,2,3,\\
&w'_k\left(0\right)=w''_k\left(0\right)=0,\quad k=1,2,3,\\
&\sum^3_{k=1}\,\bigl(w^{\left(3\right)}_k\left(0\right)+\lambda\beta\eta w_k\left(0\right)\bigr)=0.
\end{aligned}\right.
\end{equation}
Setting $\lambda=i\mu$, $\mu\in \mathbb{R}$, consider the boundary-eigenvalue problem
\begin{equation}\label{eq_1s6}
\left\{\begin{aligned}
&\varphi^{\left(4\right)}-(\gamma-\eta^2)\, \varphi''+2i\beta\eta\mu \varphi'=\mu^2 \varphi,\\
&\varphi\left(1\right)=\varphi''\left(1\right)=0,\\
&\varphi'\left(0\right)=\varphi''\left(0\right)=0.
\end{aligned}\right.
\end{equation}
Let $\varphi=\varphi\left(\lambda,s\right)$ be a nontrivial solution of \eqref{eq_1s6}. Then solutions of \eqref{302} are of the form $w_k\left(\lambda,s\right)=c_k\varphi\left(\lambda,s\right)$, $k=1,2,3$, arbitrary constants $c_k$. The continuity condition $w_j\left(0\right)=w_k\left(0\right)$, $j,k=1,2,3$, together with $\sum^3_{k=1}\,\bigl(w^{\left(3\right)}_k\left(0\right)+i\beta\eta\mu w_k\left(0\right)\bigr)=0$ implies that $c_j=c_k\equiv c$, $j,k=1,2,3$, and hence that
\begin{equation*}
3c\,\bigl(\varphi^{\left(3\right)}\left(0\right)+i\beta\eta\mu \varphi\left(0\right)\bigr)=0.
\end{equation*}
Using $\varphi^{\left(3\right)}\left(0\right)+i\beta\eta\mu \varphi\left(0\right)\ne 0$ gives $c=0$ and hence a contradiction. Thus $\mathcal{T}$ has no eigenvalues on the imaginary axis, $\operatorname{Re}\left(\lambda\right)<0$ if $\kappa>0$.
\end{proof}

\subsection{Eigenvalues and eigenvectors}

We have shown that studying the spectrum of $\mathcal{T}$ reduces to studying its discrete spectrum, consisting only of normal eigenvalues. To pursue this further, we will show in the next theorem how to determine these eigenvalues and the corresponding eigenvectors.
\begin{theorem}\label{T-3-2}
Let $\varphi=\varphi\left(\lambda,s\right)$ be a nonzero solution of the differential equation
\begin{equation}\label{303}
\varphi^{\left(4\right)} -(\gamma-\eta^2)\,\varphi''+2\lambda\beta\eta \varphi' =-\lambda^2 \varphi
\end{equation}
for $\lambda\in\mathbb{C}$ satisfying the boundary conditions
\begin{align}
\varphi \left( 1\right) = \varphi''\left(1\right) &= 0,\label{303a} \\
\varphi''\left(0\right)-\left(\alpha+\lambda\kappa\right) \varphi'\left(0\right)&=0.\label{303b}
\end{align}
Define by
\begin{equation*}
D_1\left(\lambda\right)\coloneqq\varphi^{\left(3\right)}\left(\lambda,0\right)-(\gamma-\eta^2)\,\varphi'\left(\lambda,0\right)+\lambda\beta\eta\varphi\left(\lambda,0\right),\quad D_2\left(\lambda\right)\coloneqq\varphi\left(\lambda,0\right)
\end{equation*}
the corresponding characteristic functions. The eigenvalues of $\mathcal{T}$ are the roots of $D_1\left(\lambda\right)=0$ and $D_2\left(\lambda\right)=0$,
\begin{equation*}
\sigma\left(\mathcal{T}\right)=\left\{\lambda\in \mathbb{C}~\middle|
~ D_1\left(\lambda\right)=0\right\}\cup\left\{\lambda\in \mathbb{C}~\middle|
~D_2\left(\lambda\right)=0\right\}.
\end{equation*}
If $D_1\left(\lambda\right)=0$, $D_2\left(\lambda\right)\ne 0$, then an eigenvector $x=x\left(\lambda\right)$ corresponding to the eigenvalue $\lambda$ is given by
\begin{equation}\label{306}
x\left(\lambda\right)=\bigl\{\left(\varphi\left(\lambda,\,\cdot\,\right),\lambda\varphi\left(\lambda,\,\cdot\,\right)\right);\left(\varphi\left(\lambda,\,\cdot\,\right),\lambda\varphi\left(\lambda,\,\cdot\,\right)\right); \left(\varphi\left(\lambda,\,\cdot\,\right),\lambda\varphi\left(\lambda,\,\cdot\,\right)\right)\bigr\}.
\end{equation}
If $D_1\left(\lambda\right)\ne 0$, $D_2\left(\lambda\right)=0$, then there are two linearly independent eigenvectors $x_1=x_1\left(\lambda\right)$, $x_2=x_2\left(\lambda\right)$ for the eigenvalue $\lambda$ (i.e., $\lambda$ has geometric multiplicity two) given by
\begin{align}
x_1\left(\lambda\right)&=\bigl\{\left(\varphi\left(\lambda,\,\cdot\,\right),\lambda\varphi\left(\lambda,\,\cdot\,\right)\right);-\frac{1}{2}\left(\varphi\left(\lambda,\,\cdot\,\right),\lambda\varphi\left(\lambda,\,\cdot\,\right)\right); -\frac{1}{2}\left(\varphi\left(\lambda,\,\cdot\,\right),\lambda\varphi\left(\lambda,\,\cdot\,\right)\right)\bigr\},\label{306dds}\\
x_2\left(\lambda\right)&=\bigl\{\left(0,0\right); \left(\varphi\left(\lambda,\,\cdot\,\right),\lambda\varphi\left(\lambda,\,\cdot\,\right)\right);-\left(\varphi\left(\lambda,\,\cdot\,\right),\lambda\varphi\left(\lambda,\,\cdot\,\right)\right)\bigr\}, \label{306ddsss}
\end{align}
respectively.
\end{theorem}
\begin{proof}
Let us first note from \eqref{eq_01xaa02} that if $\lambda\in\mathbb{C}$ is an eigenvalue of $\mathcal{T}$ with corresponding eigenvector $x$, then the $w_k=w_k\left(\lambda,s\right)$ are nonzero solutions of the boundary-eigenvalue problem
\begin{equation}\label{308}
\left\{\begin{aligned}
&w^{\left(4\right)}_k-(\gamma-\eta^2)\, w''_k+2\lambda\beta\eta w'_k=-\lambda^2 w_k,\quad k=1,2,3,\\
&w_k\left(1\right)=w''_k\left(1\right)=0,\quad k=1,2,3,\\
&w_j\left(0\right)=w_k\left(0\right),\quad j,k=1,2,3,\\
&w''_k\left(0\right)-\left(\alpha+\lambda\kappa\right) w'_k\left(0\right)=0,\quad k=1,2,3,\\
&\sum^3_{k=1}\,\bigl[w^{\left(3\right)}_k\left(0\right)-(\gamma-\eta^2)\,w'_k\left(0\right)+\lambda\beta\eta w_k\left(0\right)\bigr]=0.
\end{aligned}\right.
\end{equation}
Comparing \eqref{303}--\eqref{303b} with \eqref{308} we observe that $w_k\left(\lambda,s\right)=c_k\varphi\left(\lambda,s\right)$, $k=1,2,3$. Substitution in $w_j\left(0\right)=w_k\left(0\right)$, $j,k=1,2,3$, and $\sum^3_{k=1}\,\bigl[w^{\left(3\right)}_k\left(0\right)-(\gamma-\eta^2)\,w'_k\left(0\right)+\lambda\beta\eta w_k\left(0\right)\bigr]=0$ gives $c_j\varphi\left(\lambda,0\right)=c_k\varphi\left(\lambda,0\right)$, $j,k=1,2,3$, and
\begin{equation}
\sum^3_{k=1}c_k\,\bigl[\varphi^{\left(3\right)}\left(\lambda,0\right)-(\gamma-\eta^2)\,\varphi'\left(\lambda,0\right)+\lambda\beta\eta \varphi\left(\lambda,0\right)\bigr]=0. \label{310b}
\end{equation}
Two cases arise.
\begin{enumerate}[label=,leftmargin=*,align=left,labelwidth=\parindent,labelsep=0pt]
\item\label{item04y}\textbf{Case 1.} $D_2\left(\lambda\right)\ne 0$. Here $c_j=c_k\equiv c\ne 0$, $j,k=1,2,3$, since $\varphi\left(\lambda,0\right)\ne 0$, and from \eqref{310b} we have therefore
\begin{equation}\label{308v}
\varphi^{\left(3\right)}\left(\lambda,0\right)-(\gamma-\eta^2)\,\varphi'\left(\lambda,0\right)+\lambda\beta\eta \varphi\left(\lambda,0\right)=0.
\end{equation}
Then $\lambda$ is an eigenvalue if and only if the boundary-eigenvalue problem \eqref{303}--\eqref{303b}, \eqref{308v} has a nonzero solution, in which case $\lambda$ is a zero of $D_1$ and \eqref{306} follows since
\begin{equation}\label{308aaazz}
x\left(\lambda\right)=\left\{c_k\left(\varphi\left(\lambda,\,\cdot\,\right),\lambda \varphi\left(\lambda,\,\cdot\,\right)\right)\right\}^3_{k=1}
\end{equation}
with the $c_k\equiv c\ne 0$.
\item\label{item04c}\textbf{Case 2.} $D_2\left(\lambda\right)= 0$. In this case we consider \eqref{303}--\eqref{303b} subject to the additional boundary condition
\begin{equation}\label{308z}
\varphi\left(\lambda,0\right)=0.
\end{equation}
Therefore, $\lambda$ is an eigenvalue if and only if the boundary-eigenvalue problem \eqref{303}--\eqref{303b}, \eqref{308z} admits a nonzero solution. Using \eqref{308z} in \eqref{310b}, we obtain
\begin{equation*}
\sum^3_{k=1}c_k\,\bigl[\varphi^{\left(3\right)}\left(\lambda,0\right)-(\gamma-\eta^2)\,\varphi'\left(\lambda,0\right)\bigr]=\bigl[\varphi^{\left(3\right)}\left(\lambda,0\right)-(\gamma-\eta^2)\,\varphi'\left(\lambda,0\right)\bigr]\,\sum^3_{k=1}c_k=0.
\end{equation*}
Let
\begin{equation}\label{308zz}
\varphi^{\left(3\right)}\left(\lambda,0\right)-(\gamma-\eta^2)\,\varphi'\left(\lambda,0\right)=0.
\end{equation}
\end{enumerate}
It is easy to see that any solution of \eqref{303}--\eqref{303b}, \eqref{308z}, \eqref{308zz} must be the zero solution. Consequently, $\sum^3_{k=1}c_k=0$ and $\varphi^{\left(3\right)}\left(\lambda,0\right)-(\gamma-\eta^2)\,\varphi'\left(\lambda,0\right)\ne 0$, and it follows that associated with $\lambda$ there exist two linearly independent eigenvectors $x_1\left(\lambda\right)$ and $x_2\left(\lambda\right)$ given by \eqref{306dds} and \eqref{306ddsss}, respectively. The theorem is proven.
\end{proof}

\subsection{Asymptotics of eigenvalues}

Our aim in this subsection is to discuss eigenvalue asymptotics. The approach we take here is based on the ``asymptotic spectral problem'' given in \cite{MahinzaeimEtAl2022} for the single pipe case. By Theorem \ref{T-3-1}, we need only consider eigenvalues in the left half-plane $\operatorname{Re}\left(\lambda\right)\leq 0$. Moreover, those eigenvalues with nonzero imaginary part occur in conjugate pairs $\lambda$, $\overline{\lambda}$, so that we may restrict attention to $\frac{\pi}{2}\le \arg\left(\lambda\right)\le \pi$. As usual, we use the standard spectral parameter transformation $\lambda\mapsto i\rho^2$ and take $0\le \arg\left(\rho\right)\le \frac{\pi}{4}$. Define the sector ${S}$ in the complex plane by
\begin{equation*}
{S}\coloneqq\left\{\rho\in\mathbb{C}~\middle|
~0\leq \arg\left(\rho\right) \leq \frac{\pi}{4}\right\}.
\end{equation*}
For ${S}$ the four roots of $-1$ can be ordered so that
\begin{equation*}
\operatorname{Re}\left(-\rho\right)\leq \operatorname{Re}\left(i\rho\right)\leq \operatorname{Re}\left(-i\rho\right) \leq \operatorname{Re}\left(\rho\right),\quad \rho\in{S},
\end{equation*}
and from elementary considerations we have
\begin{equation*}
\operatorname{Re}\left(-\rho\right)=-\left|\rho\right|\cos\left(\arg\left(\rho\right)\right)\leq-\frac{\sqrt{2}}{2}\left|\rho\right|<0
\end{equation*}
and
\begin{equation*}
\operatorname{Re}\left(i\rho\right)=\left|\rho\right|\cos\left(\arg\left(\rho\right)+\frac{\pi}{2}\right)=-\left|\rho\right|\sin\left(\arg\left(\rho\right)\right)\leq 0.
\end{equation*}
Using this, it follows that for $\rho\in{S}$ we have as $\left|\rho\right|\rightarrow\infty$
\begin{equation*}\label{316}
\left|e^{i\rho}\right|\leq 1,\quad \left|e^{-\rho}\right|=\mathcal{O}\,(e^{-b\left|\rho\right|})\rightarrow 0,
\end{equation*}
for some constant $b>0$. We shall use these observations henceforth without explicit mention.

The underlying idea in the proof of the next theorem is that it is enough to base it on asymptotic properties in $\left|\lambda\right|$ of the solutions to \eqref{303}, according to Theorem \ref{T-3-2}. We are here concerned with the case where $\lambda$ is replaced by $i\rho^2$, and \eqref{303} becomes therefore
\begin{equation}\label{303bb}
\varphi^{\left(4\right)} -(\gamma-\eta^2)\,\varphi''+2i\beta\eta\rho^2 \varphi' =\rho^4 \varphi.
\end{equation}
For convenience, we state a result from \cite{MahinzaeimEtAl2022}.
\begin{lemma}\label{L-4-1x}
In the sector ${S}$, there exists a fundamental system $\left\{\varphi_r\left(\rho,\,\cdot\,\right)\right\}^4_{r=1}$ of the differential equation \eqref{303bb} which has the following asymptotic expressions for large $\left|\rho\right|$:
\begin{equation*}
\varphi^{\left(m\right)}_r\left(\rho,s\right)=\left(i^r\rho \right)^me^{i^r\rho s}\left(1+\Phi_r\left(s\right)+\dfrac{i^r\Phi_{r1}\left(s\right)+m\Phi'_{r}\left(s\right)}{i^r\rho}+\mathcal{O}\,(\rho^{-2})\right),\quad r=1,2,3,4,
\end{equation*}
for $m=0,1,2,3$, $0\le s\le 1$, where
\begin{equation*}
\Phi_r\left(s\right)=-1+e^{\left(-1\right)^{r+1}\frac{i\beta\eta}{2} s},\quad \Phi_{r1}\left(s\right)= \frac{\left(-i\right)^r}{4}\left(\frac{\beta^2\eta^2}{2}+\gamma-\eta^2\right) s e^{\left(-1\right)^{r+1}\frac{i\beta\eta}{2}s},\quad r=1,2,3,4.
\end{equation*}
\end{lemma}

Let us note from the lemma that, in ${S}$, since $\omega_r=i^r$ and $\omega^2_r=\left(-1\right)^r$, with $\omega_1=-\omega_3=i$ and $\omega_2=-\omega_4=-1$
\begin{equation*}
\varphi^{\left(m\right)}_r\left(\rho,s\right)=\left(\rho\omega_r\right)^me^{\rho\omega_r s}\left(1+\Phi_r\left(s\right)+\frac{\omega_r\Phi_{r1}\left(s\right)+m\Phi'_r\left(s\right)}{\rho\omega_r}+\mathcal{O}\,(\rho^{-2})\right),\quad r=1,2,3,4,
\end{equation*}
which we shall use subsequently.

We know that linear combinations of the form $\varphi\left(\rho,s\right)=\sum_{r=1}^4a_r\varphi_r\left(\rho,s\right)$ satisfy \eqref{303bb}. It is immediately obvious from Theorem \ref{T-3-2} that $D_1\left(\lambda\right)=0$ is equivalent to the condition that the system of equations
\begin{equation}\label{320}
\left\{\begin{split}
\sum^4_{r=1}a_r\varphi_r\left(\rho,1\right)&=0,\\
\sum^4_{r=1}a_r\varphi''_r\left(\rho,1\right)&=0, \\
\sum^4_{r=1}a_r\,\bigl[\varphi''_r\left(\rho,0\right)-(\alpha+i\kappa\rho^2)\, \varphi_r\left(\rho,0\right)\bigr]&=0, \\
\sum^4_{r=1}a_r\,\bigl[\varphi_r^{\left(3\right)}\left(\rho,0\right)-(\gamma-\eta^2)\, \varphi_r'\left(\rho,0\right)+i \beta\eta\rho^2 \varphi_r\left(\rho,0\right)\bigr] &=0
\end{split}\right.
\end{equation}
should have nonzero solutions. Similarly, $D_2\left(\lambda\right)=0$ in Theorem \ref{T-3-2} is equivalent to the condition that
\begin{equation}\label{321}
\left\{\begin{split}
\sum^4_{r=1}a_r\varphi_r\left(\rho,1\right)&=0,\\
\sum^4_{r=1}a_r\varphi''_r\left(\rho,1\right)&=0, \\
\sum^4_{r=1}a_r\,\bigl[\varphi''_r\left(\rho,0\right)-(\alpha+i\kappa\rho^2)\, \varphi_r\left(\rho,0\right)\bigr]&=0, \\
\sum^4_{r=1}a_r\varphi_r\left(\rho,0\right)&=0
\end{split}\right.
\end{equation}
should have nonzero solutions. Let $\Delta_i\left(\lambda\right)=\det\left(\Delta_{i,1}\left(\rho\right),\Delta_{i,2}\left(\rho\right),\Delta_{i,3}\left(\rho\right),\Delta_{i,4}\left(\rho\right)\right)$ be the characteristic determinants associated with \eqref{320} and \eqref{321}, for $i = 1, 2$, respectively, with the column entries
\begin{align*}
\Delta_{1,1}\left(\rho\right)&\coloneqq\left(\begin{matrix}
\varphi_1\left(\rho,1\right) \\[0.1em]
\varphi_1''\left(\rho,1\right) \\[0.1em]
\varphi_1''\left(\rho,0\right)-(\alpha+i\kappa\rho^2)\,\varphi_1'\left(\rho,0\right)\\[0.1em]
\varphi^{\left(3\right)}_1\left(\rho,0\right)-(\gamma-\eta^2)\,\varphi'_1\left(\rho,0\right)+i\beta\eta\rho^2 \varphi_1\left(\rho,0\right) 
\end{matrix}\right),\\[0.5em]
\Delta_{1,2}\left(\rho\right)&\coloneqq\left(\begin{matrix}
 \varphi_2\left(\rho,1\right)\\[0.1em]
 \varphi_2''\left(\rho,1\right)\\[0.1em]
 \varphi_2''\left(\rho,0\right)-(\alpha+i\kappa\rho^2)\,\varphi_2'\left(\rho,0\right)\\[0.1em]
 \varphi^{\left(3\right)}_2\left(\rho,0\right)-(\gamma-\eta^2)\,\varphi'_2\left(\rho,0\right)+i\beta\eta\rho^2 \varphi_2\left(\rho,0\right)
\end{matrix}\right),\\[0.5em]
\Delta_{1,3}\left(\rho\right)&\coloneqq\left(\begin{matrix}
 \varphi_3\left(\rho,1\right) \\[0.1em]
\varphi_3''\left(\rho,1\right) \\[0.1em]
\varphi_3''\left(\rho,0\right)-(\alpha+i\kappa\rho^2)\,\varphi_3'\left(\rho,0\right)\\[0.1em]
\varphi^{\left(3\right)}_3\left(\rho,0\right)-(\gamma-\eta^2)\,\varphi'_3\left(\rho,0\right)+i\beta\eta\rho^2 \varphi_3\left(\rho,0\right) 
\end{matrix}\right),\\[0.5em]
\Delta_{1,4}\left(\rho\right)&\coloneqq\left(\begin{matrix}
\varphi_4\left(\rho,1\right)\\[0.1em]
\varphi_4''\left(\rho,1\right)\\[0.1em]
 \varphi_4''\left(\rho,0\right)-(\alpha+i\kappa\rho^2)\,\varphi_4'\left(\rho,0\right)\\[0.1em]
 \varphi^{\left(3\right)}_4\left(\rho,0\right)-(\gamma-\eta^2)\,\varphi'_4\left(\rho,0\right)+i\beta\eta\rho^2 \varphi_4\left(\rho,0\right)
\end{matrix}\right)
\end{align*}
and
\begin{align*}
\Delta_{2,1}\left(\rho\right)&\coloneqq\left(\begin{matrix}
\varphi_1\left(\rho,1\right) \\[0.1em]
\varphi_1''\left(\rho,1\right) \\[0.1em]
\varphi_1''\left(\rho,0\right)-(\alpha+i\kappa\rho^2)\,\varphi_1'\left(\rho,0\right)\\[0.1em]
\varphi_1\left(\rho,0\right) 
\end{matrix}\right),\\[0.5em]
\Delta_{2,2}\left(\rho\right)&\coloneqq\left(\begin{matrix}
\varphi_2\left(\rho,1\right)\\[0.1em]
\varphi_2''\left(\rho,1\right)\\[0.1em]
 \varphi_2''\left(\rho,0\right)-(\alpha+i\kappa\rho^2)\,\varphi_2'\left(\rho,0\right)\\[0.1em]
 \varphi_2\left(\rho,0\right)
\end{matrix}\right),\\[0.5em]
\Delta_{2,3}\left(\rho\right)&\coloneqq\left(\begin{matrix}
\varphi_3\left(\rho,1\right) \\[0.1em]
\varphi_3''\left(\rho,1\right) \\[0.1em]
\varphi_3''\left(\rho,0\right)-(\alpha+i\kappa\rho^2)\,\varphi_3'\left(\rho,0\right)\\[0.1em]
\varphi_3\left(\rho,0\right) 
\end{matrix}\right),\\[0.5em]
\Delta_{2,4}\left(\rho\right)&\coloneqq\left(\begin{matrix}
 \varphi_4\left(\rho,1\right)\\[0.1em]
 \varphi_4''\left(\rho,1\right)\\[0.1em]
 \varphi_4''\left(\rho,0\right)-(\alpha+i\kappa\rho^2)\,\varphi_4'\left(\rho,0\right)\\[0.1em]
 \varphi_4\left(\rho,0\right)
\end{matrix}\right).
\end{align*}
Clearly the zeros of the characteristic functions $D_1$, $D_2$ coincide (including multiplicities) with those of $\Delta_1$, $\Delta_2$, respectively. So there are again two cases.
\begin{enumerate}[label=,leftmargin=*,align=left,labelwidth=\parindent,labelsep=0pt,ref={\arabic*}]
\item\label{item05y}\textbf{Case 1.} \textit{Asymptotic zeros of $\Delta_1$}. It is a straightforward calculation to show that
\begin{align*}\label{eqwetr555}
\varphi_r\left(\rho,1\right)&=e^{\rho\omega_r}e^{-\frac{i\beta\eta}{2}\omega^2_r}\\
&\qquad\times\left[1 +\frac{1}{4\rho\omega_r}\left(\frac{\beta^2\eta^2}{2}+\gamma-\eta^2 \right)+\mathcal{O}\,(\rho^{-2})\right], \\
\varphi''_r\left(\rho,1\right)&=\left(\rho\omega_r\right)^2e^{\rho\omega_r}e^{-\frac{i\beta\eta}{2}\omega^2_r}\,\biggl[1+\frac{1}{4\rho\omega_r}\,\biggl(\frac{\beta^2\eta^2}{2}+\gamma-\eta^2\\
&\qquad-4i\beta\eta \omega^2_r \biggr)+\mathcal{O}\,(\rho^{-2})\biggr],\\
\varphi''_r\left(\rho,0\right)-(\alpha+i\kappa\rho^2)\, \varphi'_r\left(\rho,0\right)&=i\kappa\rho^2\left( \rho\omega_r\right)\left(-1+\frac{i\beta\eta}{2\rho}\omega_r+\frac{\omega_r}{i\kappa\rho}+\mathcal{O}\,(\rho^{-2})\right),\\
\varphi^{\left(3\right)}_r\left(\rho,0\right)-(\gamma-\eta^2)\,\varphi'_r\left(\rho,0\right)+ i\beta\eta \rho^2 \varphi_r\left(\rho,0\right) &=\left(\rho\omega_r\right)^3\left(1-\frac{3i\beta\eta}{2\rho}\omega_r+\frac{i\beta\eta}{\rho}\omega_r+\mathcal{O}\,(\rho^{-2})\right)
\end{align*}
for $r=1,2,3,4$. Substituting these expressions in the equation $\Delta_1\left(\lambda\right)=0$ and performing elementary computations, we obtain that  $\Delta_{1,r}\left(\rho\right)=\Delta_{1,r,0}\left(\rho\right)+\mathcal{O}\,(e^{-b\left|\rho\right|})$, $r=1,2,3,4$, where
\begin{align*}
\Delta_{1,1,0}\left(\rho\right)&\coloneqq\left(\begin{matrix}
e^{i\rho}e^{\frac{i\beta\eta}{2}}\left[1+\frac{1}{4i\rho}\left(\frac{\beta^2\eta^2}{2}+\gamma-\eta^2 \right)+\mathcal{O}\,(\rho^{-2})\right] \\[0.1em]
-e^{i\rho}e^{\frac{i\beta\eta}{2}}\left[1+\frac{1}{4i\rho}\left(\frac{\beta^2\eta^2}{2}+\gamma-\eta^2 +4i\beta\eta\right)+\mathcal{O}\,(\rho^{-2})\right] \\[0.1em]
-\kappa\left(-1-\frac{\beta\eta}{2\rho}+\frac{1}{\kappa\rho}+\mathcal{O}\,(\rho^{-2})\right) \\[0.1em]
-i\left(1+\frac{\beta\eta}{2\rho}+\mathcal{O}\,(\rho^{-2})\right) 
\end{matrix}\right),\\[0.5em]
\Delta_{1,2,0}\left(\rho\right)&\coloneqq\left(\begin{matrix}
 0 \\[0.1em]
 0\\[0.1em]
 -i\kappa\left(-1-\frac{i\beta\eta}{2\rho}-\frac{1}{i\kappa\rho}+\mathcal{O}\,(\rho^{-2})\right)\\[0.1em]
 -\left(1+\frac{i\beta\eta}{2\rho}+\mathcal{O}\,(\rho^{-2})\right)
\end{matrix}\right),\\[0.5em]
\Delta_{1,3,0}\left(\rho\right)&\coloneqq\left(\begin{matrix}
e^{-i\rho}e^{\frac{i\beta\eta}{2}}\left[1-\frac{1}{4i\rho}\left(\frac{\beta^2\eta^2}{2}+\gamma-\eta^2 \right)+\mathcal{O}\,(\rho^{-2})\right]\\[0.1em]
-e^{-i\rho}e^{\frac{i\beta\eta}{2}}\left[1-\frac{1}{4i\rho}\left(\frac{\beta^2\eta^2}{2}+\gamma-\eta^2 +4i\beta\eta\right)+\mathcal{O}\,(\rho^{-2})\right]\\[0.1em]
\kappa\left(-1+\frac{\beta\eta}{2\rho}-\frac{1}{\kappa\rho}+\mathcal{O}\,(\rho^{-2})\right) \\[0.1em]
i\left(1-\frac{\beta\eta}{2\rho}+\mathcal{O}\,(\rho^{-2})\right) 
\end{matrix}\right),\\[0.5em]
\Delta_{1,4,0}\left(\rho\right)&\coloneqq\left(\begin{matrix}
 e^{-\frac{i\beta\eta}{2}}\left[1+\frac{1}{4\rho}\left(\frac{\beta^2\eta^2}{2}+\gamma-\eta^2 \right)+\mathcal{O}\,(\rho^{-2})\right]\\[0.1em]
 e^{-\frac{i\beta\eta}{2}}\left[1+\frac{1}{4\rho}\left(\frac{\beta^2\eta^2}{2}+\gamma-\eta^2 -4i\beta\eta\right)+\mathcal{O}\,(\rho^{-2})\right]\\[0.1em]
 0\\[0.1em]
 0
\end{matrix}\right)
\end{align*}
and we have used that $\omega_1=-\omega_3=i$, $\omega_2=-\omega_4=-1$, $\omega_1^2=\omega_3^2=-1$, $\omega_2^2=\omega_4^2=1$ and therefore that $e^{\rho\omega_2}=e^{-\rho\omega_4}=e^{-\rho}$, $e^{\rho\omega_1}=e^{-\rho\omega_3}=e^{
i\rho}$ in ${S}$. So the characteristic equation has an asymptotic representation of the form
\begin{equation}\label{eq22wcharw}
\cos\rho+\left(\frac{\beta^2\eta^2}{2}+\gamma-\eta^2\right)\frac{\sin\rho+\cos\rho}{4\rho}-\frac{i\left(\cos\rho-\sin\rho\right)}{2\kappa\rho}+\mathcal{O}\,(\rho^{-2})=0
\end{equation}
or
\begin{equation*}
\cos\rho+\mathcal{O}\,(\rho^{-1})=0.
\end{equation*}
(The idea of the proof would now be to establish, by applying Rouche's theorem, that, choosing $f\left(\rho\right)= \cos\rho$ and $g\left(\rho\right)= -\mathcal{O}\,(\rho^{-1})$,  $\left|g\right|\le 1< \left|f\right|$ on a closed contour in the complex plane around $\rho=\left(n+1/2\right)\pi$ for large $n$, and $f$ must therefore have only one zero within the contour; see \cite[Theorem 3.2]{MahinzaeimEtAl2022}.) Let the sequence $\left\{\lambda_n\right\}$ represent the roots of \eqref{eq22wcharw} and set $\rho_n=\left(n+\frac{1}{2}\right)\pi+z_n$.
Since
\begin{align*}
\cos\rho_n&=\cos\left(n+\frac{1}{2}\right)\pi\cos z_n-\sin\left(n+\frac{1}{2}\right)\pi\sin z_n=-\left(-1\right)^{n}\sin z_n,\\
\sin\rho_n&=\sin\left(n+\frac{1}{2}\right)\pi\cos z_n+\cos\left(n+\frac{1}{2}\right)\pi\sin z_n=\left(-1\right)^{n}\cos z_n,
\end{align*}
it follows that the $z_n$ satisfy
\begin{equation*}
\sin z_n=\left(\frac{i}{\kappa}+\frac{\frac{\beta^2\eta^2}{2}+\gamma-\eta^2}{2}\right)\frac{\cos z_n}{2\rho_n}+\mathcal{O}\,(\rho^{-2}_n).
\end{equation*}
Thus, for large $n$ (equivalently, small $z_n$),
\begin{equation*}
z_n=\left(\frac{i}{\kappa}+\frac{\frac{\beta^2\eta^2}{2}+\gamma-\eta^2}{2}\right)\frac{1}{2\left(n+\frac{1}{2}\right)\pi}+\mathcal{O}\,(n^{-2}).
\end{equation*}
We write $\rho_n=\tau_n+z_n$ where $\tau_n= \left(n+\frac{1}{2}\right)\pi$ to obtain, taking into account $\lambda_n=i\rho_n^2=i\left(\tau_n+z_n\right)^2$,
\begin{equation*}
\lambda_n=-\frac{1}{\kappa}+i\left(\tau_n^2+ \frac{\frac{\beta^2\eta^2}{2}+\gamma-\eta^2
 }{2}\right)+\mathcal{O}\,(\tau_n^{-1}),\quad\tau_n=\left(n+\frac{1}{2}\right)\pi.
\end{equation*}
\item\label{item05yss}\textbf{Case 2.} \textit{Asymptotic zeros of $\Delta_2$}. In this case,
\begin{align*}
\varphi_r\left(\rho,1\right)&=e^{\rho\omega_r}e^{-\frac{i\beta\eta}{2}\omega^2_r}\\
&\qquad\times\left[1 +\frac{1}{4\rho\omega_r}\left(\frac{\beta^2\eta^2}{2}+\gamma-\eta^2 \right)+\mathcal{O}\,(\rho^{-2})\right], \\
\varphi''_r\left(\rho,1\right)&=\left(\rho\omega_r\right)^2e^{\rho\omega_r}e^{-\frac{i\beta\eta}{2}\omega^2_r}\,\biggl[1+\frac{1}{4\rho\omega_r}\,\biggl(\frac{\beta^2\eta^2}{2}+\gamma-\eta^2\\
&\qquad-4i\beta\eta \omega^2_r \biggr)+\mathcal{O}\,(\rho^{-2})\biggr],\\
\varphi''_r\left(\rho,0\right)-(\alpha+i\kappa\rho^2)\, \varphi'_r\left(\rho,0\right)&=i\kappa\rho^2\left( \rho\omega_r\right)\left(-1+\frac{i\beta\eta}{2\rho}\omega_r+\frac{\omega_r}{i\kappa\rho}+\mathcal{O}\,(\rho^{-2})\right),\\
\varphi_r\left(\rho,0\right) &=1+\mathcal{O}\,(\rho^{-2})
\end{align*}
for $r=1,2,3,4$ and hence, using arguments analogous to those given in Case \ref{item05y} we compute $\Delta_{2,r}\left(\rho\right)=\Delta_{2,r,0}\left(\rho\right)+\mathcal{O}\,(e^{-b\left|\rho\right|})$, $r=1,2,3,4$, where
\begin{align*}
\Delta_{2,1,0}\left(\rho\right)&\coloneqq\left(\begin{matrix}
e^{i\rho}e^{\frac{i\beta\eta}{2}}\left[1+\frac{1}{4i\rho}\left(\frac{\beta^2\eta^2}{2}+\gamma-\eta^2 \right)+\mathcal{O}\,(\rho^{-2})\right] \\[0.1em]
-e^{i\rho}e^{\frac{i\beta\eta}{2}}\left[1+\frac{1}{4i\rho}\left(\frac{\beta^2\eta^2}{2}+\gamma-\eta^2 +4i\beta\eta\right)+\mathcal{O}\,(\rho^{-2})\right]
\\[0.1em]
-\kappa\left(-1-\frac{\beta\eta}{2\rho}+\frac{1}{\kappa\rho}+\mathcal{O}\,(\rho^{-2})\right)\\[0.1em]
1+\mathcal{O}\,(\rho^{-2})
\end{matrix}\right),\\[0.5em]
\Delta_{2,2,0}\left(\rho\right)&\coloneqq\left(\begin{matrix}
 0 \\[0.1em]
 0\\[0.1em]
-i\kappa\left(-1-\frac{i\beta\eta}{2\rho}-\frac{1}{i\kappa\rho}+\mathcal{O}\,(\rho^{-2})\right)\\[0.1em]
1+\mathcal{O}\,(\rho^{-2})
\end{matrix}\right),\\[0.5em]
\Delta_{2,3,0}\left(\rho\right)&\coloneqq\left(\begin{matrix}
e^{-i\rho}e^{\frac{i\beta\eta}{2}}\left[1-\frac{1}{4i\rho}\left(\frac{\beta^2\eta^2}{2}+\gamma-\eta^2 \right)+\mathcal{O}\,(\rho^{-2})\right] \\[0.1em]
 -e^{-i\rho}e^{\frac{i\beta\eta}{2}}\left[1-\frac{1}{4i\rho}\left(\frac{\beta^2\eta^2}{2}+\gamma-\eta^2 +4i\beta\eta\right)+\mathcal{O}\,(\rho^{-2})\right]\\[0.1em]
\kappa\left(-1+\frac{\beta\eta}{2\rho}-\frac{1}{\kappa\rho}+\mathcal{O}\,(\rho^{-2})\right)\\[0.1em]
1+\mathcal{O}\,(\rho^{-2})
\end{matrix}\right),\\[0.5em]
\Delta_{2,4,0}\left(\rho\right)&\coloneqq\left(\begin{matrix}
e^{-\frac{i\beta\eta}{2}}\left[1+\frac{1}{4\rho}\left(\frac{\beta^2\eta^2}{2}+\gamma-\eta^2 \right)+\mathcal{O}\,(\rho^{-2})\right]\\[0.1em]
 e^{-\frac{i\beta\eta}{2}}\left[1+\frac{1}{4\rho}\left(\frac{\beta^2\eta^2}{2}+\gamma-\eta^2 -4i\beta\eta\right)+\mathcal{O}\,(\rho^{-2})\right]\\[0.1em]
0\\[0.1em]
0
\end{matrix}\right),
\end{align*}
and therefore the characteristic equation has an asymptotic representation of the form
\begin{equation}\label{eq22wchar}
\cos\left(\rho+\frac{\pi}{4}\right)+\left(\frac{\beta^2\eta^2}{2}+\gamma-\eta^2 \right)\frac{\sqrt{2}\cos\rho}{4\rho} +\frac{2i\sqrt{2}\sin\rho}{2\kappa\rho}+\mathcal{O}\,(\rho^{-2})=0
\end{equation}
or
\begin{equation*}
\cos\left(\rho+\frac{\pi}{4}\right)+\mathcal{O}\,(\rho^{-1})=0,
\end{equation*}
where we have used that $\frac{\sqrt{2}}{2}\cos\rho-\frac{\sqrt{2}}{2}\sin\rho=\cos\left(\rho+\frac{\pi}{4}\right)$. Let the sequence $\left\{\lambda_n\right\}$ represent the roots of \eqref{eq22wchar} and set $\rho_n=\left(n+\frac{1}{4}\right)\pi+z_n$. Since
\begin{align*}
\cos\rho_n&=\cos\left(n+\frac{1}{4}\right)\pi\cos z_n-\sin\left(n+\frac{1}{4}\right)\pi\sin z_n=\frac{\sqrt{2}}{2}\left(-1\right)^{n}\cos z_n-\frac{\sqrt{2}}{2}\left(-1\right)^{n}\sin z_n,\\
\sin\rho_n&=\sin\left(n+\frac{1}{4}\right)\pi\cos z_n+\cos\left(n+\frac{1}{4}\right)\pi\sin z_n=\frac{\sqrt{2}}{2}\left(-1\right)^{n}\cos z_n+\frac{\sqrt{2}}{2}\left(-1\right)^{n}\sin z_n,
\end{align*}
we have
\begin{equation*}
\cos\left(\rho_n+\frac{\pi}{4}\right)=-\left(-1\right)^{n}\sin z_n
\end{equation*}
and, proceeding as before, that the $z_n$ in this case satisfy
\begin{equation*}
\sin z_n=\left(\frac{\frac{\beta^2\eta^2}{2}+\gamma-\eta^2}{2}\right)\frac{\cos z_n}{2\rho_n}+\frac{i}{\kappa\rho_n}+\mathcal{O}\,(\rho^{-2}_n).
\end{equation*}
Therefore, for large $n$,
\begin{equation*}
z_n=\left(\frac{\frac{\beta^2\eta^2}{2}+\gamma-\eta^2}{2}\right)\frac{1}{2\left(n+\frac{1}{4}\right)\pi}+\frac{i}{\kappa\left(n+\frac{1}{4}\right)\pi}+\mathcal{O}\,(n^{-2}),
\end{equation*}
and by a similar calculation to Case \ref{item05y}, we then get
\begin{equation*}
\lambda_n=-\frac{2}{\kappa}+i\left(\tau_n^2+ \frac{\frac{\beta^2\eta^2}{2}+\gamma-\eta^2
 }{2}\right)+\mathcal{O}\,(\tau_n^{-1}),\quad \tau_n=\left(n+\frac{1}{4}\right)\pi.
\end{equation*}
\end{enumerate}

We collect the foregoing results together into the following theorem.
\begin{theorem}\label{T-3-3}
The spectrum of $\mathcal{T}$ consists of two branches of a discrete set of eigenvalues, $\sigma\left(\mathcal{T}\right)=\sigma^{\left(1\right)}\left(\mathcal{T}\right)\cup\sigma^{\left(2\right)}\left(\mathcal{T}\right)$, where asymptotically, for large enough $n$,
\begin{equation*}\label{305}
\sigma^{\left(1\right)}\left(\mathcal{T}\right)=\left\{\lambda\in \mathbb{C}~\middle|~ \Delta_1\left(\lambda\right)=0\right\}=\{\lambda^{\left(1\right)}_{\pm n}{\}}^\infty_{n=0},\quad 
 \sigma^{\left(2\right)}\left(\mathcal{T}\right)=\left\{\lambda\in \mathbb{C}~\middle|~ \Delta_2\left(\lambda\right)=0\right\}=\{\lambda^{\left(2\right)}_{\pm n}{\}}^\infty_{n=0},
\end{equation*}
with the sequences enumerated properly in the sense of \cite[Remark 4.2]{MahinzaeimEtAl2021} (i.e., indexed in such a way that $\lambda_n=\overline{\lambda_{-n}}$ whenever $\operatorname{Im}\left({\lambda}_n\right)\neq 0$). Both branches are confined to vertical strips in the open left half-plane such that, as $n\rightarrow\infty$,
\begin{equation*}
|\operatorname{Re}\,(\lambda^{(i)}_{ n})|\leq N<\infty,\quad \operatorname{Im}\,(\lambda^{(i)}_{ n})\rightarrow \infty,\quad i=1,2,
\end{equation*}
some constant $N$. The eigenvalues belonging to $\sigma_1\left(\mathcal{T}\right)$ have asymptotic representations
\begin{equation*}
\lambda^{\left(1\right)}_{ \pm n}=-\frac{1}{\kappa}\pm i\left((\tau^{\left(1\right)}_n{)}^2+ \frac{\frac{\beta^2\eta^2}{2}+\gamma-\eta^2
 }{2}\right)+\mathcal{O}\,\bigl((\tau^{\left(1\right)}_{ n }{)}^{-1}\bigr),\quad \tau^{\left(1\right)}_n=\left(n+\frac{1}{2}\right)\pi,\quad n\rightarrow\infty,
\end{equation*}
and those belonging to $\sigma_2\left(\mathcal{T}\right)$ have asymptotic representations
\begin{equation*}
\lambda^{\left(2\right)}_{\pm n}=-\frac{2}{\kappa}\pm i\left((\tau^{\left(2\right)}_n{)}^2+ \frac{\frac{\beta^2\eta^2}{2}+\gamma-\eta^2
 }{2}\right)+\mathcal{O}\,\bigl((\tau^{\left(2\right)}_{ n }{)}^{-1}\bigr),\quad \tau^{\left(2\right)}_n=\left(n+\frac{1}{4}\right)\pi,\quad n\rightarrow\infty.
\end{equation*}
\end{theorem}

\subsection{Multiplicity of Eigenvalues}

For the study of eigenvalue multiplicities, we quote the following useful result from \cite[Corollary 4.2.2]{Xu2010}.
\begin{lemma}\label{xumult}
Let $\mathcal{A}$ be a linear operator in a Hilbert space $\mathscr{X}$, and let $\lambda$ be an eigenvalue of $\mathcal{A}$ with corresponding eigenvector $x$, i.e.\ $\left(\lambda {I}-\mathcal{A}\right)x=0$, $x\neq 0$. If there is a nonzero element $z\in\operatorname{Ker} \,(\overline{\lambda} {I}-\mathcal{A}^*)$ such that the inner product $\left(x,z\right) \neq 0$, then $\lambda$ is a simple (respectively, semisimple) eigenvalue if $\operatorname{dim}\operatorname{Ker} \left(\lambda {I}-\mathcal{A}\right)=1$ (respectively, $\operatorname{dim}\operatorname{Ker} \left(\lambda {I}-\mathcal{A}\right)\ne 1$).
\end{lemma}

In order to use Lemma \ref{xumult}, we need to consider the spectral problem for the adjoint operator for $\mathcal{T}$, $\mathcal{T}^*$. First we prove a proposition.
\begin{proposition}\label{L-3-1}
The adjoint operator $\mathcal{T}^*$ is defined on the domain
\begin{equation}\label{314}
\mathscr{D}\left(\mathcal{T}^*\right)=\left\{z=\left\{z_k\right\}^3_{k=1}\in \mathscr{X}~\middle|
~\begin{gathered}
z_k=\left(\widetilde{w}_k,\widetilde{v}_k\right)\in (\bm{H}^4\left(0,1\right)\cap \bm{H}_*^2\left(0,1\right))\times \bm{H}_*^2\left(0,1\right),\\
\widetilde{w}''_k\left(1\right)=0,\quad \widetilde{w}''_k\left(0\right)-\alpha \widetilde{w}'_k\left(0\right)+\kappa \widetilde{v}'_k\left(0\right)=0,\quad k=1,2,3,\\
\sum_{k=1}^3\,\bigl[\widetilde{w}^{\left(3\right)}_k\left(0\right)-(\gamma-\eta^2)\, \widetilde{w}_{k}'\left(0\right)+\beta\eta \widetilde{v}_k\left(0\right)\bigr]=0
\end{gathered}\right\},
\end{equation}
by
\begin{equation}\label{313}
\mathcal{T}z\coloneqq\bigl\{\bigl(-\widetilde{v}_k,\widetilde{w}_k^{\left(4\right)}-(\gamma-\eta^2)\,\widetilde{w}''_k+2\beta\eta \widetilde{v}_k'\bigr)\bigr\}_{k=1}^3.
\end{equation}
\end{proposition}
\begin{proof}
The proof is a formal calculation. First note that $x\in\mathscr{D}\left(\mathcal{T}\right)\left(=\mathscr{D}\left(\mathcal{A}\right)\right)$ implies $v_k\left(1\right)=0$, $k=1,2,3$, and $v_j\left(0\right)=v_k\left(0\right)$, $j,k=1,2,3$. Using this, it follows from integration by parts that if $x\in\mathscr{D}\left(\mathcal{T}\right)$, then for any $z=\left\{\left(\widetilde{w}_k,\widetilde{v}_k\right)\right\}^3_{k=1}\in \bm{H}^4\left(\mathbf{G}\right)\times \bm{H}^2\left(\mathbf{G}\right)\left(\supset \mathscr{D}\left(\mathcal{T}^*\right)\right)$,
\begin{align*}
\left(\mathcal{T}x,z\right)_\mathscr{X}&=-\sum_{k=1}^3\left[\int^1_0w''_k\left(s\right)\overline{\widetilde{v}''_k\left(s\right)}ds+(\gamma-\eta^2) \int^1_0w'_k\left(s\right)\overline{\widetilde{v}'_k\left(s\right)}ds+\alpha w'_k\left(0\right)\overline{\widetilde{v}'_k\left(0\right)}\right]\\
&\qquad+\sum_{k=1}^3\int^1_0v_k\left(s\right)\left[\overline{\widetilde{w}^{\left(4\right)}_k\left(s\right)}-(\gamma-\eta^2) \,\overline{\widetilde{w}''_k\left(s\right)} +2\beta\eta \overline{\widetilde{v}'_k\left(s\right)}\right]ds\\
&\qquad+\sum_{k=1}^3v_k'\left(1\right)\overline{\widetilde{w}_k''\left(1\right)}-\sum_{k=1}^3\biggl[w^{\left(3\right)}_k\left(1\right)-(\gamma-\eta^2)\,w'_k\left(1\right)+2\beta\eta v_k\left(1\right)\biggr]\,\overline{\widetilde{v}_k\left(1\right)}\\
&\qquad-\sum_{k=1}^3v_k'\left(0\right)\Bigl(\overline{\widetilde{w}_k''\left(0\right)}-\alpha\overline{\widetilde{w}_k'\left(0\right)}+\kappa\widetilde{v}'_k\left(0\right)\Bigr)\\
&\qquad+\sum_{k=1}^3v_k\left(0\right)\left[\overline{\widetilde{w}^{\left(3\right)}_k\left(0\right)}-(\gamma-\eta^2)\,\overline{\widetilde{w}_k'\left(0\right)}+\beta\eta \widetilde{v}'_k\left(0\right)\right]\\
&\qquad-\sum_{k=1}^3\Bigl(w''_k\left(0\right)-\alpha w'_k\left(0\right)-\kappa v'_k\left(0\right)\Bigr)\, \overline{\widetilde{v}'_k\left(0\right)}\\
&\qquad+\sum_{k=1}^3\biggl[w^{\left(3\right)}_k\left(0\right)- (\gamma-\eta^2)\,w'_k\left(0\right)+\beta\eta v_k\left(0\right)\biggr]\,\overline{\widetilde{v}_k\left(0\right)}
\end{align*}
Let $\widetilde{v}_k\left(1\right)=0$, $k=1,2,3$, and let $\widetilde{v}_j\left(0\right)=\widetilde{v}_k\left(0\right)$, $j,k=1,2,3$. Further, let $\widetilde{w}_k\left(1\right)=\widetilde{w}''_k\left(1\right)=0$, $\widetilde{w}''_k\left(0\right)-\alpha \widetilde{w}'_k\left(0\right)+\kappa \widetilde{v}'_k\left(0\right)=0$, $k=1,2,3$, $\widetilde{w}_j\left(0\right)=\widetilde{w}_k\left(0\right)$, $j,k=1,2,3$, and $\sum^3_{k=1}\,\bigl[\widetilde{w}^{\left(3\right)}_k\left(0\right)-(\gamma-\eta^2)\, \widetilde{w}_{k}'\left(0\right)+\beta\eta \widetilde{v}_k\left(0\right)\bigr]=0$. Then we have $\left\{\left(\widetilde{w}_k,\widetilde{v}_k\right)\right\}^3_{k=1}\in \mathscr{D}\left(\mathcal{T}^*\right)$ and
\begin{align*}
\left(\mathcal{T}x,z\right)_\mathscr{X}&=\left(x,\mathcal{T}^* z\right)_\mathscr{X}\\
&=-\sum_{k=1}^3\left[\int^1_0w''_k\left(s\right)\overline{\widetilde{v}''_k\left(s\right)}ds+(\gamma-\eta^2) \int^1_0w'_k\left(s\right)\overline{\widetilde{v}'_k\left(s\right)}ds+\alpha w'_k\left(0\right)\overline{\widetilde{v}'_k\left(0\right)}\right]\\
&\qquad+\sum_{k=1}^3\int^1_0v_k\left(s\right)\left[\overline{\widetilde{w}^{\left(4\right)}_k\left(s\right)}-(\gamma-\eta^2) \,\overline{\widetilde{w}''_k\left(s\right)} +2\beta\eta \overline{\widetilde{v}'_k\left(s\right)}\right]ds\\
\end{align*}
where $\mathscr{D}\left(\mathcal{T}^*\right)$, $\mathcal{T}^*$ are as in the lemma, completing the proof.
\end{proof}
Let us now consider the spectral problem for $\mathcal{T}^*$ defined by \eqref{314}, \eqref{313},
\begin{equation}\label{eq_01xaa02s}
\mathcal{T}^*z=\mu z,\quad z\in\mathscr{D}\left(\mathcal{T}^*\right),\quad \mu\in\mathbb{C},
\end{equation}
which in coordinates is given by
\begin{equation}\label{213ss}
\left\{\begin{aligned}
&\widetilde{w}^{\left(4\right)}_k-(\gamma-\eta^2)\, \widetilde{w}''_k-2\mu\beta \eta \widetilde{w}_{k}'=-\mu^2\widetilde{w}_k,\quad k=1,2,3,\\
&\widetilde{w}_k\left(1\right)=\widetilde{w}''_k\left(1\right)=0,\quad k=1,2,3,\\
&\widetilde{w}_j\left(0\right)=\widetilde{w}_k\left(0\right),\quad j,k=1,2,3,\\
&\widetilde{w}''_k\left(0\right)-\left(\alpha+\mu\kappa\right) \widetilde{w}'_k\left(0\right)=0,\quad k=1,2,3,\\
&\sum^3_{k=1}\,\bigl[\widetilde{w}^{\left(3\right)}_k\left(0\right)-(\gamma-\eta^2)\, \widetilde{w}_{k}'\left(0\right)+\mu\beta\eta \widetilde{w}_k\left(0\right)\bigr]=0,
\end{aligned}\right.
\end{equation}
where we have taken into account that $\widetilde{v}_k=-\mu\widetilde{w}_k$, $k=1,2,3$. Proceeding formally as in the proof of Theorem \ref{T-3-2} we let $\psi=\psi\left(\mu,s\right)$ be a nonzero solution of the differential equation
\begin{equation}\label{eq12sd44a}
\psi^{\left(4\right)}-(\gamma-\eta^2)\, \psi''-2\mu\beta \eta \psi'=-\mu^2\psi
\end{equation}
satisfying the boundary conditions
\begin{align}
\psi\left(1\right)=\psi''\left(1\right) &= 0,\label{eq12sd44b} \\
\psi''\left(0\right)-\left(\alpha+\mu\kappa\right) \psi'\left(0\right)&=0.\label{eq12sd44c}
\end{align}
Obviously, $\psi\,(\overline{\mu},\,\cdot\,)=\overline{\psi\,(\mu,\,\cdot\,)}$. It is clear that solutions $w_k=w_k\left(\mu,s\right)$ of \eqref{213ss} are of the form $w_k\left(\mu,s\right)=b_k\psi\left(\mu,s\right)$, $k=1,2,3$, arbitrary constants $b_k$, and therefore we have in the case $\psi\left(\mu,0\right)\ne 0$ that $\mu$ is an eigenvalue if and only if \eqref{eq12sd44a}--\eqref{eq12sd44c} with the supplementary requirement
\begin{equation*}
\psi^{\left(3\right)}\left(\mu,0\right)-(\gamma-\eta^2)\,\psi'\left(\mu,0\right)-\mu \beta\eta\psi\left(\mu,0\right)=0
\end{equation*}
has a nonzero solution. The corresponding eigenvector $z=z\left(\mu\right)$ is given by
\begin{equation}\label{30ddd3xa}
z\left(\mu\right)=\left\{b_k\left(\psi\left(\mu,\,\cdot\,\right),-\mu \psi\left(\mu,\,\cdot\,\right)\right)\right\}^3_{k=1}
\end{equation}
with the $b_k\equiv b\ne 0$.

In the case $\psi\left(\mu,0\right)=0$, again continuing as in the proof of Theorem \ref{T-3-2}, we have $\sum^3_{k=1}b_k=0$ and it follows that there are two linearly independent eigenvectors for the eigenvalue $\mu$, $z_1=z_1\left(\mu\right)$, $z_2=z_2\left(\mu\right)$, given by
\begin{align}
z_1\left(\mu\right)&=\bigl\{\left(\psi\left(\mu,\,\cdot\,\right),-\mu\psi\left(\mu,\,\cdot\,\right)\right);-\frac{1}{2}\left(\psi\left(\mu,\,\cdot\,\right),-\mu\psi\left(\mu,\,\cdot\,\right)\right); -\frac{1}{2}\left(\psi\left(\mu,\,\cdot\,\right),-\mu\psi\left(\mu,\,\cdot\,\right)\right)\bigr\},\label{306ddsa}\\
z_2\left(\mu\right)&=\bigl\{\left(0,0\right); \left(\psi\left(\mu,\,\cdot\,\right),-\mu\psi\left(\mu,\,\cdot\,\right)\right);-\left(\psi\left(\mu,\,\cdot\,\right),-\mu\psi\left(\mu,\,\cdot\,\right)\right)\bigr\},\label{306ddsssa}
\end{align}
respectively. We can now turn to the verification of the condition $\left(x,z\right)_{\mathscr{X}} \neq 0$ in Lemma \ref{xumult}. To do this note first that, since $\mathscr{X}$ is a Hilbert space, $\sigma\left(\mathcal{T}^*\right)=\overline{\sigma\left(\mathcal{T}\right)}$. Let $\lambda\in\sigma\left(\mathcal{T}\right)$, $\mu\in\sigma\left(\mathcal{T}^*\right)$, and let $x\left(\lambda\right)=\left\{\left(w_k\left(\lambda,\,\cdot\,\right),v_k\left(\lambda,\,\cdot\,\right)\right)\right\}_{k=1}^3$ and $z\left(\mu\right)=\left\{\left(\widetilde{w}_k\left(\mu,\,\cdot\,\right),\widetilde{v}_k\left(\mu,\,\cdot\,\right)\right)\right\}_{k=1}^3$ be the eigenvectors corresponding to $\lambda$ and $\mu$, respectively. We take the inner product of \eqref{eq_01xaa02} with $z\left(\mu\right)$ and obtain, taking into account \eqref{eq_01xaa02s},
\begin{equation*}
\lambda\left( x\left(\lambda\right),z\left(\mu\right)\right)_\mathscr{X}=\left( \mathcal{T}x\left(\lambda\right),z\left(\mu\right)\right)_\mathscr{X}=\left( x\left(\lambda\right),\mathcal{T}^*z\left(\mu\right)\right)_\mathscr{X}=\overline{\mu}\left( x\left(\lambda\right),z\left(\mu\right)\right)_\mathscr{X}.
\end{equation*}
Consequently, for $\mu\neq\overline{\lambda}$, $\left( x\left(\lambda\right),z\left(\mu\right)\right)_\mathscr{X}=0$, so we set $\mu=\overline{\lambda}$. Direct calculations show that
\begin{equation}\label{sumeq12}
(x\left(\lambda\right),z\,(\overline{\lambda}){)}_{\mathscr{X}}=B\left(\lambda\right)\sum_{k=1}^3c_k \overline{b_k}
\end{equation}
with
\begin{align*}
B\left(\lambda\right)&\coloneqq \int^1_0\varphi''\left(\lambda,s\right)\psi''\left(\lambda,s\right) ds+(\gamma-\eta^2)\int^1_0\varphi'\left(\lambda,s\right)\psi'\left(\lambda,s\right) ds\\
&\qquad+\alpha\varphi'\left(\lambda,0\right)\psi'\left(\lambda,0\right)-\lambda^2\int_{0}^1\varphi\left(\lambda,s\right)\psi\left(\lambda,s\right)ds,
\end{align*}
which is easily verified on using the fact that, by \eqref{308aaazz} and \eqref{30ddd3xa}, $w_k\left(\lambda,\,\cdot\,\right)=c_k\varphi\left(\lambda,\,\cdot\,\right)$, $v_k\left(\lambda,\,\cdot\,\right)=c_k\lambda \varphi\left(\lambda,\,\cdot\,\right)$, $\widetilde{w}_k\left(\lambda,\,\cdot\,\right)=b_k\psi\,(\overline{\lambda},\,\cdot\,)$, $\widetilde{v}_k\left(\lambda,\,\cdot\,\right)=-b_k\overline{\lambda} \psi\,(\overline{\lambda},\,\cdot\,)$, $k=1,2,3$, and that $\overline{\psi\,(\overline{\lambda},\,\cdot\,)}=\psi\,(\lambda,\,\cdot\,)$. Now integration by parts in the first and second integrals in $B\left(\lambda\right)$ yields
\begin{align*}
B\left(\lambda\right)&=-\lambda^2\int_{0}^1\varphi\left(\lambda,s\right)\psi\left(\lambda,s\right)ds+\int_0^1\biggl[\varphi^{\left(4\right)}\left(\lambda,s\right)-(\gamma-\eta^2)\, \varphi''\left(\lambda,s\right)\biggr]\,\psi\left(\lambda,s\right)ds\\
&\qquad+\varphi''\left(\lambda,1\right)\psi'\left(\lambda,1\right)-\left[\varphi^{\left(3\right)}\left(\lambda,1\right)
-(\gamma-\eta^2)\,\varphi'\left(\lambda,1\right)\right]\psi\left(\lambda,1\right)\\
&\qquad-\left(\varphi''\left(\lambda,0\right)-\alpha\varphi'\left(\lambda,0\right)\right) \psi'\left(\lambda,0\right)
+\left[\varphi^{\left(3\right)}\left(\lambda,0\right) -(\gamma-\eta^2)\,\varphi'\left(\lambda,0\right)\right]\psi\left(\lambda,0\right).
\end{align*}
Inserting the differential equation \eqref{303} and using that $\varphi''\left(1\right)=0$, $\varphi''\left(0\right)-\left(\alpha+\lambda\kappa\right)\varphi\left(0\right)=0$, $\psi\left(1\right)=0$, we can write
\begin{align*}
B\left(\lambda\right)&=\left(D_1\left(\lambda\right)-\lambda\beta\eta D_2\left(\lambda\right)\right)\psi\left(\lambda,0\right)
-\lambda\kappa \varphi'\left(\lambda,0\right)\psi'\left(\lambda,0\right)\\
&\qquad-2\lambda\int_{0}^1\left(\lambda\varphi\left(\lambda,s\right)+\beta\eta
\varphi'\left(\lambda,s\right)\right)\psi\left(\lambda,s\right)ds.
\end{align*}
Define $(\partial {\varphi}/\partial \lambda)\left(\lambda,s\right)=\widehat{\varphi}\left(\lambda,s\right)$. Then, assuming that $\widehat{\varphi}$ is an associated function, we have from \eqref{303}--\eqref{303b} that $\widehat{\varphi}=\widehat{\varphi}\left(\lambda,s\right)$ satisfies the differential equation
\begin{equation}\label{303xa}
\widehat{\varphi}^{\left(4\right)} -(\gamma-\eta^2)\,\widehat{\varphi}''+2\lambda\beta\eta \widehat{\varphi}'+\lambda^2 \widehat{\varphi} =-2\beta\eta {\varphi}'-2\lambda \varphi
\end{equation}
and the boundary conditions
\begin{align}
\widehat{\varphi} \left( 1\right) = \widehat{\varphi}''\left(1\right) &= 0,\label{303axb} \\
\widehat{\varphi}''\left(0\right)-\left(\alpha+\lambda\kappa\right) \widehat{\varphi}'\left(0\right)-\kappa\varphi'\left(0\right)&=0.\label{303bxc}
\end{align}
Recalling the expressions for the characteristic functions $D_1$, $D_2$ from Theorem \ref{T-3-2}, we also have
\begin{align}
D'_1\left(\lambda\right)&=\widehat{\varphi}^{\left(3\right)}\left(\lambda,0\right)-(\gamma-\eta^2)\,\widehat{\varphi}'\left(\lambda,0\right)+\lambda\beta\eta\widehat{\varphi}\left(\lambda,0\right)+\beta\eta{\varphi}\left(\lambda,0\right),\label{303axsb}\\
 D'_2\left(\lambda\right)&=\widehat{\varphi}\left(\lambda,0\right)\label{303asxb}
\end{align}
where the primes over the characteristic functions indicate differentiation with respect to $\lambda$. By using \eqref{303xa}--\eqref{303asxb}, it follows from integration by parts that with \eqref{eq12sd44a}--\eqref{eq12sd44c},
\begin{align*}
B\left(\lambda\right)&=\left(D_1\left(\lambda\right)-\lambda D'_1\left(\lambda\right)\right)\psi\left(\lambda,0\right)+\lambda D'_2\left(\lambda\right)\left[\psi^{(3)}\left(\lambda,0\right)-(\gamma-\eta^2)\,\psi'\left(\lambda,0\right)-\lambda\beta\eta\psi\left(\lambda,0\right)\right].
\end{align*}
If $D_1\left(\lambda\right)=0$, then $D_2\left(\lambda\right)\ne0$. This implies, since $\psi\left(\lambda,0\right)\ne 0$ and $\psi^{(3)}\left(\lambda,0\right)-(\gamma-\eta^2)\,\psi'\left(\lambda,0\right)-\lambda\beta\eta\psi\left(\lambda,0\right)=0$,
\begin{equation*}
B\left(\lambda\right)=-\lambda D'_1\left(\lambda\right)\psi\left(\lambda,0\right).
\end{equation*}
If $D_2\left(\lambda\right)=0$, then $D_1\left(\lambda\right)\ne0$, which implies
\begin{equation*}
B\left(\lambda\right)=\lambda D'_2\left(\lambda\right)\left[\psi^{(3)}\left(\lambda,0\right)-(\gamma-\eta^2)\,\psi'\left(\lambda,0\right)\right],
\end{equation*}
since $\psi\left(\lambda,0\right)= 0$ and $\psi^{(3)}\left(\lambda,0\right)-(\gamma-\eta^2)\,\psi'\left(\lambda,0\right)\ne 0$.
So clearly if $\lambda$ belongs to the branch $\sigma_1\left(\mathcal{T}\right)$ of the spectrum $\sigma\left(\mathcal{T}\right)$, then from \eqref{sumeq12} we have
\begin{equation*}
(x\left(\lambda\right),z\,(\overline{\lambda}){)}_{\mathscr{X}}=3bcB\left(\lambda\right)\neq 0,
\end{equation*}
as $bc\ne0$. If $\lambda$ belongs to the branch $\sigma_2\left(\mathcal{T}\right)$ of $\sigma\left(\mathcal{T}\right)$ we have
\begin{equation*}
(x_1\left(\lambda\right),z_1\,(\overline{\lambda}){)}_{\mathscr{X}}=\frac{3}{2}B\left(\lambda\right)\neq 0,\quad (x_2\left(\lambda\right),z_2\,(\overline{\lambda}){)}_{\mathscr{X}}=2B\left(\lambda\right)\neq 0
\end{equation*}
where the eigenvectors $x_i\left(\lambda\right)$ and $z_i\,(\overline{\lambda})$ have the forms given by \eqref{306dds}, \eqref{306ddsa} and \eqref{306ddsss}, \eqref{306ddsssa}, for $i = 1, 2$, respectively. The calculations above imply that associated with an eigenvector $x$ corresponding to an eigenvalue $\lambda\in\sigma_1\left(\mathcal{T}\right)$ there is a nontrivial element $z\in \operatorname{Ker} \left(\overline{\lambda}{I}-\mathcal{T}^*\right)$ such that $\left( x,z\right)_\mathscr{X}\neq0$. Then each $\lambda\in\sigma_1\left(\mathcal{T}\right)$ is simple since, by Theorem \ref{T-3-2}, it has a one-dimensional geometric eigenspace, $\dim \operatorname{Ker} \left(\lambda{I}-\mathcal{T}\right)=1$. Similarly, it also follows from the calculations above that associated with an eigenvector $x$ corresponding to an eigenvalue $\lambda\in\sigma_2\left(\mathcal{T}\right)$ there is a nontrivial element $z\in \operatorname{Ker} \left(\overline{\lambda}{I}-\mathcal{T}^*\right)$ such that $\left( x,z\right)_\mathscr{X}\neq0$. Hence, we obtain that each $\lambda\in\sigma_2\left(\mathcal{T}\right)$ is semisimple and has multiplicity $2$ since $\dim \operatorname{Ker} \left(\lambda{I}-\mathcal{T}\right)=2$. We have proven the following theorem.
\begin{theorem}\label{T-3-4}
All eigenvalues of $\mathcal{T}$ are semisimple, i.e.\ there are no associated vectors corresponding to any eigenvalue $\lambda\in\sigma\left(\mathcal{T}\right)$. Each eigenvalue belonging to $\sigma_1\left(\mathcal{T}\right)$ has multiplicity $1$, and each eigenvalue belonging to $\sigma_2\left(\mathcal{T}\right)$ has multiplicity $2$.
\end{theorem}

\section{Completeness, minimality, and Riesz basis properties of eigenvectors}\label{sec_04}

We begin by collecting three familiar operator results from the literature which we require in the sequel. The first is due to \cite[Theorem V.8.1]{GohbergKrein1969}, the second to \cite[Lemma 2.4]{MalamudEtAl2012}, and the third to \cite[Theorem 1.1]{xu2005expansion}.
\begin{lemma}\label{thm04abc}
Let $\mathcal{K}$ be a compact skewadjoint operator on a Hilbert space $\mathscr{X}$ with $\operatorname{Ker} \mathcal{K}=\left\{0\right\}$, and let $\mathcal{S}$ be a real operator on $\mathscr{X}$ which has finite rank. Let
\begin{equation*}
\mathcal{Q}=\mathcal{K}+\kappa \mathcal{S},\quad \kappa\geq 0.
\end{equation*}
Then the root vectors of the operator $\mathcal{Q}$ are complete in $\mathscr{X}$.
\end{lemma}

\begin{lemma}\label{thm04abcd}
Let $\mathcal{Q}$ be a compact operator on $\mathscr{X}$ and $\operatorname{Ker} \mathcal{Q}=\left\{0\right\}$. Then the root vectors of $\mathcal{Q}$ are minimal in $\mathscr{X}$.
\end{lemma}

\begin{lemma}\label{T-6-4}
Let $\mathscr{X}$ be a separable Hilbert space and let $\mathcal{Q}$ be the infinitesimal generator of a $C_0$-semigroup $\mathbb{U}\left(t\right)$ on $\mathscr{X}$. Suppose that the following conditions hold:
\begin{enumerate}[\normalfont(i)]
\item\label{T-6-4-a} $\sigma\left(\mathcal{Q}\right)=\sigma^{(1)}\left(\mathcal{Q}\right)\cup\sigma^{(2)}\left(\mathcal{Q}\right)$ where the branch $\sigma^{(2)}\left(\mathcal{Q}\right)=\{\lambda_k{\}}^\infty_{k=1}$, consisting entirely of isolated eigenvalues of finite algebraic multiplicity;
\item\label{T-6-4-b} $\sup_{k\geq 1}m_a\left(\lambda_k\right)<\infty$ with $m_a\left(\lambda_k\right)\coloneqq\dim E\left(\lambda_k,\mathcal{Q}\right)\mathscr{X}$, the $E\left(\lambda_k,\mathcal{Q}\right)$ being the Riesz projections associated with the eigenvalues $\lambda_k$; and
\item\label{T-6-4-c} there is a real constant $\nu$ such that
\begin{equation*}
\sup\,\{\operatorname{Re}\left(\lambda\right)~|~\lambda\in\sigma^{(1)}\left(\mathcal{Q}\right)\}\leq \nu\leq \inf\,\{\operatorname{Re}\left(\lambda\right)~|~\lambda\in\sigma^{(2)}\left(\mathcal{Q}\right)\},
\end{equation*}
and
\begin{equation}\label{eq12vsep}
\inf_{k\neq j}\left|\lambda_k-\lambda_j\right|>0.
\end{equation}
\end{enumerate}
Then the following statements hold:
\begin{enumerate}[\normalfont(1)]
\item\label{T-6-4-d} There exist two $\mathbb{U}\left(t\right)$-invariant closed subspaces $\mathscr{X}_1$ and $\mathscr{X}_2$ with the properties
\begin{enumerate}[\normalfont(i)]
\item $\sigma\left(\mathcal{Q}|_{\mathscr{X}_1}\right)=\sigma^{(1)}\left(\mathcal{Q}\right)$ and $\sigma\left(\mathcal{Q}|_{\mathscr{X}_2}\right)=\sigma^{(2)}\left(\mathcal{Q}\right)$; and
\item $\left\{E\left(\lambda_k,\mathcal{Q}\right)\mathscr{X}_2\right\}^{\infty}_{k=1}$ forms a Riesz basis of subspaces for $\mathscr{X}_2$, and
\begin{equation*}
\mathscr{X}=\overline{{\mathscr{X}_1\oplus \mathscr{X}_2}}.
\end{equation*}
\end{enumerate}
\item\label{T-6-4-e} If $\sup_{k\geq 1}\left\|E\left(\lambda_k,\mathcal{Q}\right)\right\|<\infty$, then
\begin{equation*}
\mathscr{D}\left(\mathcal{Q}\right)\subset \mathscr{X}_1\oplus \mathscr{X}_2\subset \mathscr{X}.
\end{equation*}
\item\label{T-6-4-f} $\mathscr{X}$ can be decomposed into the topological direct sum
\begin{equation*}
\mathscr{X}=\mathscr{X}_1\oplus \mathscr{X}_2
\end{equation*}
if and only if $\sup_{n\geq 1}\,\bigl\|\sum\limits^n_{k=1}E\left(\lambda_k,\mathcal{Q}\right)\bigr\|<\infty$.
\end{enumerate}
\end{lemma}

\subsection{Completeness and minimality}
We shall now show that the root vectors of $\mathcal{T}$, all of which shown in Theorem \ref{T-3-4} to be eigenvectors, are minimal complete in $\mathscr{X}$. The result can be obtained by application of a similar method to the one in \cite{MahinzaeimEtAl2021} to construct the inverse operator $\mathcal{T}^{-1}$, which we know by Lemma \ref{L-2-1} to exist and be compact, and subsequent application of Lemmas \ref{thm04abc} and \ref{thm04abcd}. To this end, we need first of all the following theorem.

\begin{theorem}\label{T-4-1}
Let $\mathcal{T}_{0}$ be the skewadjoint part of $\mathcal{T}$, i.e., $\mathcal{T}\coloneqq \mathcal{A}+\mathcal{B}$ as defined by \eqref{204}--\eqref{205} with $\kappa=0$. Then $\mathcal{T}^{-1}=\mathcal{T}_{0}^{-1}+\kappa \mathcal{S}$, where $\mathcal{S}$ is a real operator on $\mathscr{X}$ which is of finite rank.
\end{theorem}
\begin{proof}
Consider the problem \eqref{213}. If we integrate the differential equation twice from $0$ to $1$, making use of the vertex conditions $w_k \left(1\right)=w_k''\left(1\right) = 0$, $k=1,2,3$, we arrive at
\begin{equation}\label{402}
\begin{split}
&w_k'' \left(s\right)-(\gamma-\eta^2)\,w_k \left(s\right)+\left[w_{k}^{\left(3\right)}\left(0\right)-(\gamma-\eta^2)\, w_{k}'\left(0\right)+\beta \eta v_k\left(0\right)\right]\\
&\hspace{0.4\linewidth}\times\left(1-s\right)=-\widetilde{V}_{k}\left(s\right)-\widetilde{W}_{k}\left(s\right),\quad k=1,2,3,
\end{split}
\end{equation}
with the integral terms
\begin{equation*}
\widetilde{V}_{k}\left(s\right)= -\int^1_sdt\int^t_0\widetilde{v}_k\left(r\right)dr, \quad
\widetilde{W}_{k}\left(s\right)= -2\beta \eta\int^1_sdt\int^t_0\widetilde{w}'_k\left(r\right)dr-\beta\eta\int^1_s\widetilde{w}_k\left(0\right)dt.
\end{equation*}
For brevity, let us put $\widetilde{\left(V;W\right)}_k \left(s\right) \coloneqq\widetilde{V}_k \left(s\right)+\widetilde{W}_k\left(s\right)$ in this proof. The solutions of \eqref{402} then take the form
\begin{equation}\label{403}
\begin{split}
w_k \left(s\right)&= c_k \sinh\sqrt{\gamma-\eta^2} \left(1-s\right)+\left[w_{k}^{\left(3\right)}\left(0\right)-(\gamma-\eta^2) \,w_{k}'\left(0\right)+\beta \eta v_k\left(0\right)\right]\\
&\qquad\times\frac{1}{\sqrt{\gamma-\eta^2}}\int^1_s\left(1-r\right)\sinh\sqrt{\gamma-\eta^2}\left(s-r\right) dr\\
&\qquad+\frac{1}{\sqrt{\gamma-\eta^2}}\int^1_s\sinh\sqrt{\gamma-\eta^2}\left(s-r\right)\widetilde{\left(V;W\right)}_k \left(r\right) dr,\quad k=1,2,3,
\end{split}
\end{equation}
with arbitrary constants $c_k$. Thus the following equations are obtained
\begin{equation*}
\begin{split}
w_k \left(0\right)&= c_k \sinh\sqrt{\gamma-\eta^2}-\left[w_{k}^{\left(3\right)}\left(0\right)-(\gamma-\eta^2) \,w_{k}'\left(0\right)+\beta \eta v_k\left(0\right)\right]\\
&\qquad\times\frac{1}{\sqrt{\gamma-\eta^2}}\int^1_0\left(1-r\right)\sinh\sqrt{\gamma-\eta^2} \, r \,dr\\
&\qquad-\frac{1}{\sqrt{\gamma-\eta^2}}\int^1_0\sinh\sqrt{\gamma-\eta^2} \, r\,\widetilde{\left(V;W\right)}_k \left(r\right) dr,\\
w_k' \left(0\right)&= -c_k \sqrt{\gamma-\eta^2}\cosh\sqrt{\gamma-\eta^2}+\left[w_{k}^{\left(3\right)}\left(0\right)-(\gamma-\eta^2)\, w_{k}'\left(0\right)+\beta \eta v_k\left(0\right)\right]\\
&\qquad\times\int^1_0\left(1-r\right)\cosh\sqrt{\gamma-\eta^2} \, r\, dr+\int^1_0\cosh\sqrt{\gamma-\eta^2} \,r\,\widetilde{\left(V;W\right)}_k \left(r\right) dr,\\
w_k'' \left(0\right)&= c_k\,(\gamma-\eta^2) \sinh\sqrt{\gamma-\eta^2}-\left[w_{k}^{\left(3\right)}\left(0\right)-(\gamma-\eta^2)\, w_{k}'\left(0\right)+\beta \eta v_k\left(0\right)\right]\\
&\qquad\times\left[\sqrt{\gamma-\eta^2}\int^1_0\left(1-r\right)\sinh\sqrt{\gamma-\eta^2} \, r \,dr+1\right]\\
&\qquad -\sqrt{\gamma-\eta^2}\int^1_0\sinh\sqrt{\gamma-\eta^2} \, r\,\widetilde{\left(V;W\right)}_k \left(r\right) dr-\widetilde{\left(V;W\right)}_k \left(0\right)
\end{split}
\end{equation*}
for $k=1,2,3$. If we sum each of these three equations over $k=1,2,3$, using the force balance condition $\sum^3_{k=1}\,\bigl[w^{\left(3\right)}_k\left(0\right)-(\gamma-\eta^2)\, w_{k}'\left(0\right)+\beta\eta v_k\left(0\right)\bigr]=0$, we obtain
\begin{equation}\label{407a}
\left\{\begin{split}
\sum_{k=1}^3w_k \left(0\right)&= \sinh\sqrt{\gamma-\eta^2}\,\sum_{k=1}^3c_k -\frac{1}{\sqrt{\gamma-\eta^2}}\int^1_0\sinh\sqrt{\gamma-\eta^2} \,r\sum_{k=1}^3\widetilde{\left(V;W\right)}_k \left(r\right) dr\\
\sum_{k=1}^3w_k' \left(0\right)&= -\sqrt{\gamma-\eta^2}\cosh\sqrt{\gamma-\eta^2}\sum_{k=1}^3c_k
+\int^1_0\cosh\sqrt{\gamma-\eta^2} \, r\sum_{k=1}^3\widetilde{\left(V;W\right)}_k \left(r\right) dr\\
\sum_{k=1}^3w_k'' \left(0\right)&= (\gamma-\eta^2)\, \sinh\sqrt{\gamma-\eta^2}\sum_{k=1}^3c_k -\sqrt{\gamma-\eta^2}\int^1_0\sinh\sqrt{\gamma-\eta^2} \, r\sum_{k=1}^3\widetilde{\left(V;W\right)}_k \left(r\right) dr\\
&\qquad-\sum_{k=1}^3\widetilde{\left(V;W\right)}_k \left(0\right).
\end{split}\right.
\end{equation}
Using the vertex condition $w''_k\left(0\right)-\alpha w'_k\left(0\right)-\kappa v'_k\left(0\right)=0$, $k=1,2,3$, taking into account that the $v_k=\widetilde{w}_k$, and summing over $k=1,2,3$ we find, again using $\sum^3_{k=1}\,\bigl[w^{\left(3\right)}_k\left(0\right)-(\gamma-\eta^2)\, w_{k}'\left(0\right)+\beta\eta v_k\left(0\right)\bigr]=0$, that
\begin{align*}
\sum_{k=1}^3c_k&=\frac{\sum_{k=1}^3\left[\widetilde{\left(V;W\right)}_k \left(0\right)
+\int^1_0\bigl(\sqrt{\gamma-\eta^2}\sinh\sqrt{\gamma-\eta^2} \,r+\alpha\cosh\sqrt{\gamma-\eta^2}\, r\bigr)\,\widetilde{\left(V;W\right)}_k \left(r\right) dr\right]}{(\gamma-\eta^2)\, \sinh\sqrt{\gamma-\eta^2}+\alpha\sqrt{\gamma-\eta^2}\cosh\sqrt{\gamma-\eta^2}}\\
&\qquad+\frac{\kappa\sum_{k=1}^3 \widetilde{w}_{k}'\left(0\right)
}{(\gamma-\eta^2)\, \sinh\sqrt{\gamma-\eta^2}+\alpha\sqrt{\gamma-\eta^2}\cosh\sqrt{\gamma-\eta^2}}\\
&\coloneqq b\left(\widetilde{w},\widetilde{v}\right)+\kappa H\left(\widetilde{w}\right),
\end{align*}
where
\begin{equation*}
 H\left(\widetilde{w}\right)\coloneqq\frac{\sum_{k=1}^3\widetilde{w}_{k}'\left(0\right)}{(\gamma-\eta^2)\, \sinh\sqrt{\gamma-\eta^2}+\alpha\sqrt{\gamma-\eta^2}\cosh\sqrt{\gamma-\eta^2}}.
\end{equation*}
Clearly the first equation in \eqref{407a} can be written as
\begin{equation*}
\sum_{k=1}^3w_k \left(0\right)= \left(b\left(\widetilde{w},\widetilde{v}\right)+\kappa H\left(\widetilde{w}\right)\right) \sinh\sqrt{\gamma-\eta^2} -\frac{1}{\sqrt{\gamma-\eta^2}}\int^1_0\sinh\sqrt{\gamma-\eta^2} \, r\sum_{k=1}^3\widetilde{\left(V;W\right)}_k \left(r\right) dr.
\end{equation*}
Since $w_j\left(0\right)=w_k\left(0\right)\equiv w\left(0\right)$, $j,k=1,2,3$, it follows that
\begin{align*}
w\left(0\right)&=\frac{1}{3}\left[\left(b\left(\widetilde{w},\widetilde{v}\right)+\kappa H\left(\widetilde{w}\right)\right) \sinh\sqrt{\gamma-\eta^2} -\frac{1}{\sqrt{\gamma-\eta^2}}\int^1_0\sinh\sqrt{\gamma-\eta^2}r \sum_{k=1}^3\widetilde{\left(V;W\right)}_k \left(r\right) dr\right]\\
&\coloneqq c\left(\widetilde{w},\widetilde{v}\right)+\kappa G\left(\widetilde{w}\right)
\end{align*}
where
\begin{equation*}
G\left(\widetilde{w}\right)\coloneqq \frac{H\left(\widetilde{w}\right) }{3} \sinh\sqrt{\gamma-\eta^2} .
\end{equation*}
Thus we have
\begin{equation*}
\sum_{k=1}^3c_k=b\left(\widetilde{w},\widetilde{v}\right)+\kappa H\left(\widetilde{w}\right),\quad
w\left(0\right)=c\left(\widetilde{w},\widetilde{v}\right)+\kappa G\left(\widetilde{w}\right),
\end{equation*}
and hence the following algebraic equations for $c_k$, $w^{\left(3\right)}_k\left(0\right)-(\gamma-\eta^2)\, w_{k}'\left(0\right)+\beta\eta v_k\left(0\right)$:
\begin{equation}\label{408}
\left\{\begin{split}
&c_k \sinh\sqrt{\gamma-\eta^2} -\frac{w_{k}^{\left(3\right)}\left(0\right)-(\gamma-\eta^2)\, w_{k}'\left(0\right)+\beta \eta v_k\left(0\right)}{\sqrt{\gamma-\eta^2}}\int^1_0\left(1-r\right)\sinh\sqrt{\gamma-\eta^2} \, r\, dr\\
&\hspace{0.2\linewidth}=c\left(\widetilde{w},\widetilde{v}\right)+\frac{1}{\sqrt{\gamma-\eta^2}}\int^1_0\sinh\sqrt{\gamma-\eta^2} \, r\,\widetilde{\left(V;W\right)}_k \left(r\right) dr,\\
&c_k\left(\sinh\sqrt{\gamma-\eta^2}+\frac{\alpha \cosh\sqrt{\gamma-\eta^2}}{\sqrt{\gamma-\eta^2}} \right)-\frac{w_{k}^{\left(3\right)}\left(0\right)-(\gamma-\eta^2)\, w_{k}'\left(0\right)+\beta \eta v_k\left(0\right)}{\sqrt{\gamma-\eta^2}}\\
&\hspace{0.1\linewidth}\times\left[\int^1_0\left(1-r\right)\left(\sinh\sqrt{\gamma-\eta^2} \, r+\frac{\alpha\cosh\sqrt{\gamma-\eta^2} \, r}{\sqrt{\gamma-\eta^2}}\right) dr+1 \right]\\
&\hspace{0.2\linewidth}=\frac{1}{\sqrt{\gamma-\eta^2}}\int^1_0\left(\sinh\sqrt{\gamma-\eta^2} \, r+
\frac{\alpha\cosh\sqrt{\gamma-\eta^2}\,r}{\sqrt{\gamma-\eta^2}} \right)\widetilde{\left(V;W\right)}_k \left(r\right) dr\\
&\hspace{0.2\linewidth}\qquad+\frac{\widetilde{\left(V;W\right)}_k \left(0\right)+\kappa \widetilde{w}_{k}'\left(0\right)}{\gamma-\eta^2}
\end{split}\right.
\end{equation}
for $k=1,2,3$. We observe that the determinant of the coefficient matrix formed by \eqref{408} does not vanish,
\begin{equation*}
\left|\begin{matrix}
 \sinh\sqrt{\gamma-\eta^2}& -\int^1_0\left(1-r\right)\sinh\sqrt{\gamma-\eta^2} \, r\, dr\\
\alpha\cosh\sqrt{\gamma-\eta^2}
&-\alpha\int^1_0\left(1-r\right)\cosh\sqrt{\gamma-\eta^2} \, r \, dr-1 \\
\end{matrix}\right|\neq 0,
\end{equation*}
whence there are unique solution pairs $c_k$, $w^{\left(3\right)}_k\left(0\right)-(\gamma-\eta^2)\, w_{k}'\left(0\right)+\beta\eta v_k\left(0\right)$ of the form
\begin{equation*}
\begin{aligned}
c_k&=c_k\left(\widetilde{w},\widetilde{v}\right)+\kappa d_k\left(\widetilde{w}\right),& k&=1,2,3,\\
w_{k}^{\left(3\right)}\left(0\right)-(\gamma-\eta^2)\, w_{k}'\left(0\right)+\beta \eta v_k\left(0\right)&=M_k\left(\widetilde{w}, \widetilde{v}\right)+\kappa N_k\left(\widetilde{w}\right),& k&=1,2,3.
\end{aligned}
\end{equation*}
Inserting these into \eqref{403}, we find
\begin{align*}
w_k \left(s\right)&= c_k\left(\widetilde{w},\widetilde{v}\right) \sinh\sqrt{\gamma-\eta^2} \left(1-s\right) + \frac{M_k\left(\widetilde{w},\widetilde{v}\right)}{\sqrt{\gamma-\eta^2}}\int^1_s\left(1-r\right)\sinh\sqrt{\gamma-\eta^2}\left(s-r\right) dr\\
&\qquad+\frac{1}{\sqrt{\gamma-\eta^2}}\int^1_s\sinh\sqrt{\gamma-\eta^2}\left(s-r\right)\widetilde{\left(V;W\right)}_k \left(r\right) dr\\
&\qquad+\kappa\left[d_k\left(\widetilde{w}\right) \sinh\sqrt{\gamma-\eta^2} \left(1-s\right) +\frac{N_k\left(\widetilde{w}\right)}{\sqrt{\gamma-\eta^2}}\int^1_s\left(1-r\right)\sinh\sqrt{\gamma-\eta^2}\left(s-r\right) dr\right]\\
&\coloneqq w_k \left(s,\widetilde{w},\widetilde{v}\right)+\kappa \left(d_k\left(\widetilde{w}\right)\phi_1 \left(s\right)+N_k\left(\widetilde{w}\right)\phi_2\left(s\right)\right)
\end{align*}
where $\phi_1\left(s\right)=\sinh\sqrt{\gamma-\eta^2} \left(1-s\right) $, $\phi_2\left(s\right)=\frac{1}{\sqrt{\gamma-\eta^2}}\int^1_s\left(1-r\right)\sinh\sqrt{\gamma-\eta^2}\left(s-r\right) dr$. Therefore we have
\begin{align*}
\mathcal{T}^{-1}\left\{\left(\widetilde{w}_k,\widetilde{v}_k\right)\right\}_{k=1}^3&=\left\{\left(w_k,v_k\right)\right\}_{k=1}^3\\
&=\left\{\left(w_k\left(s,\widetilde{w},\widetilde{v}\right),v_k\right)\right\}_{k=1}^3+\kappa\left\{\left(d_k\left(\widetilde{w}\right)\phi_1+N_k\left(\widetilde{w}\right)\phi_2,0\right)\right\}_{k=1}^3\\
&=\mathcal{T}_{0}^{-1}\left\{\left(\widetilde{w}_k,\widetilde{v}_k\right)\right\}_{k=1}^3+\kappa \mathcal{S}\left\{\left(\widetilde{w}_k,\widetilde{v}_k\right)\right\}_{k=1}^3,
\end{align*}
where $\mathcal{T}^{-1}_0$ is compact and skewadjoint and $\mathcal{S}$ is of rank $2$. In particular if $\mathcal{T}x=y$, $y\in \mathscr{X}$, $x\in \mathscr{D}\left(\mathcal{A}\right)$, it is easily seen that
\begin{equation*}
\operatorname{Re} \left( \mathcal{T}x,x\right)_\mathscr{X} = \operatorname{Re}\, ( \mathcal{T}^{-1}y,y{)}_\mathscr{X} = \operatorname{Re} \,( (\mathcal{T}_0^{-1}+\kappa \mathcal{S})\,y,y{)}_\mathscr{X} = \kappa \operatorname{Re} \,( \mathcal{S}y,y{)}_\mathscr{X} =\kappa \,( \mathcal{S}y,y{)}_\mathscr{X}.
 \end{equation*}
This completes the proof of the theorem.
\end{proof}
Combined with Lemmas \ref{thm04abc} and \ref{thm04abcd}, if we identify $\mathcal{Q}$ with $\mathcal{T}^{-1}$ (recall from Lemma \ref{L-2-1} that $0\in\varrho\left(\mathcal{T}\right)$ and $\mathcal{T}^{-1}$ is compact) and note that the geometric eigenspaces of $\mathcal{Q}$ and $\mathcal{T}$ corresponding, respectively, to eigenvalues $\lambda$ and $\lambda^{-1}$ coincide, Theorem \ref{T-4-1} leads to the following result.
\begin{theorem}\label{T-4-2}
The eigenvectors of $\mathcal{T}$ are minimal complete in $\mathscr{X}$.
\end{theorem}

\subsection{Riesz basis property}
There remains the verification of the Riesz basis property of the eigenvectors. We use Lemma \ref{T-6-4} to prove the following theorem.

\begin{theorem}\label{T-4-3}
There exists a sequence of eigenvectors corresponding to a properly enumerated sequence of eigenvalues of $\mathcal{T}$ which is a Riesz basis for $\mathscr{X}$.
\end{theorem}
\begin{proof}
We identify the operator $\mathcal{Q}$ in Lemma \ref{T-6-4} with $\mathcal{T}$ and begin by noting that, by Theorem \ref{T-3-4}, for any $\lambda\in\sigma\left(\mathcal{T}\right)$, $E\left(\lambda,T\right)\mathscr{X}$ is equal to the geometric eigenspace $\operatorname{Ker} \left(\lambda{I}-\mathcal{T}\right)$ and so $\dim E\left(\lambda, T\right)\mathscr{X}=1$ or $\dim E\left(\lambda, T\right)\mathscr{X}=2$. Let us take $\sigma^{\left(1\right)}\left(\mathcal{T}\right)=\left\{-\infty\right\}$ and $\sigma^{\left(2\right)}\left(\mathcal{T}\right)=\sigma\left(\mathcal{T}\right)$. It now follows from Theorem \ref{T-3-3} that all conditions of Lemma \ref{T-6-4} are satisfied. In particular, by virtue of Theorem \ref{T-3-3}, all but at most finitely many of the eigenvalues are simple for large enough $n$, i.e., the eigenvalue sequences are interpolating (because, asymptotically, the eigenvalues are uniformly bounded away from the imaginary axis and the uniform gap condition \eqref{eq12vsep} holds). It also follows from Theorem \ref{T-4-2} that
\begin{equation*}
\overline{\operatorname{span}\left\{E\left(\lambda,\mathcal{T}\right)\mathscr{X}~\middle|~\lambda\in\sigma^{\left(2\right)}\left(\mathcal{T}\right)\right\}}=\mathscr{X}_2=\mathscr{X}.
\end{equation*}
Hence, combining these results, we have by Lemma \ref{T-6-4} that there exists a sequence of eigenvectors of $\mathcal{T}$ which is a Riesz basis for $\mathscr{X}$.
\end{proof}

As a corollary of the theorem we obtain the following result.
\begin{corollary}\label{T-4-4}
$\mathcal{T}$ satisfies the Spectrum Determined Growth Assumption.
\end{corollary}

\section{Vertex feedback stabilisability result}\label{sec_05}

We now have all the ingredients to apply the spectral approach to obtain the final result of the paper.
\begin{theorem}\label{T-5-1}
If $\kappa>0$, then there exist constants $M\ge 1$ and $\varepsilon>0$ such that
\begin{equation*}
\left\|\mathbb{S}\left(t\right)x_0\right\|_\mathscr{X}\le Me^{-\varepsilon t}\left\|x_0\right\|_\mathscr{X},\quad x_0\in\mathscr{D}\left(\mathcal{T}\right),\quad t\in \mathbf{R}_+,
\end{equation*}
where $\mathbb{S}\left(t\right)$ is the $C_0$-semigroup of contractions generated by $\mathcal{T}\coloneqq \mathcal{A}+\mathcal{B}$ defined by \eqref{204}--\eqref{205}. Consequently we have for the solutions of the $\mathfrak{CL}$-system, as $t\rightarrow\infty$, $\left\|{x}\left(t\right)\right\|_\mathscr{X}\rightarrow 0$ exponentially and the $\mathfrak{CL}$-system is thus vertex feedback stabilisable.
\end{theorem}
\begin{proof}
Since, by Corollary \ref{T-4-4}, $\mathcal{T}$ satisfies the Spectrum Determined Growth Assumption, we have $\varpi_0=\sup\left\{\operatorname{Re}\left(\lambda\right)~\middle|~\lambda\in\sigma\left(\mathcal{T}\right)\right\}$. It follows from Theorem \ref{T-3-1} that $i\mathbb{R}\subset\varrho\left(\mathcal{T}\right)$, and according to Theorem \ref{T-3-3} the two branches $\sigma^{\left(1\right)}\left(\mathcal{T}\right)$, $\sigma^{\left(2\right)}\left(\mathcal{T}\right)$ have asymptotes $\operatorname{Re}\left(\lambda\right)\sim-\frac{1}{\kappa}$, $\operatorname{Re}\left(\lambda\right)\sim-\frac{2}{\kappa}$, respectively. Thus $\sup\left\{\operatorname{Re}\left(\lambda\right)~\middle|~\lambda\in\sigma\left(\mathcal{T}\right)\right\}<0$ and hence $\sup\left\{\operatorname{Re}\left(\lambda\right)~\middle|~\lambda\in\sigma\left(\mathcal{T}\right)\right\}\leq -\varepsilon<0$ for some $\varepsilon>0$. With this the proof is complete.
\end{proof}

\bigskip\noindent
\textbf{Acknowledgments.} The research described here was supported in part by the Stiftung KESSLER+CO für Bildung und Kultur und Grant EXPLOR-24MM

\bibliographystyle{plain}
\bibliography{BibLio02}

\begin{thebibliography}{10}

\bibitem{MR3109894}
F.~Abdallah, D.~Mercier, and S.~Nicaise.
\newblock Exponential stability of the wave equation on a star shaped network
  on indefinite sign damping.
\newblock {\em Palest. J. Math.}, 2:113--143, 2013.

\bibitem{AissaEtAl2021}
A.~B. Aissa, M.~Abdelli, and A.~Duca.
\newblock Well-posedness and exponential decay for the {E}uler--{B}ernoulli
  beam conveying fluid equation with non-constant velocity and dynamical
  boundary conditions.
\newblock {\em Z.\ Angew.\ Math.\ Phys.}, 72:1--15, 2021.

\bibitem{MR3799048}
K.~Ammari and E.~Cr\'{e}peau.
\newblock Feedback stabilization and boundary controllability of the
  {K}orteweg--de {V}ries equation on a star-shaped network.
\newblock {\em SIAM J. Control Optim.}, 56:1620--1639, 2018.

\bibitem{MR4501388}
K.~Ammari and F.~Hassine.
\newblock {\em Stabilization of {K}elvin--{V}oigt damped systems}.
\newblock Birkh\"{a}user, 2022.

\bibitem{MR3891267}
K.~Ammari and F.~Shel.
\newblock Stability of a tree-shaped network of strings and beams.
\newblock {\em Math. Methods Appl. Sci.}, 41:7915--7935, 2018.

\bibitem{MR4397494}
K.~Ammari and F.~Shel.
\newblock {\em Stability of elastic multi-link structures}.
\newblock Springer, 2022.

\bibitem{MR1814271}
K.~Ammari and M.~Tucsnak.
\newblock Stabilization of {B}ernoulli--{E}uler beams by means of a pointwise
  feedback force.
\newblock {\em SIAM J. Control Optim.}, 39:1160--1181, 2000.

\bibitem{MR4592978}
A.~Bchatnia, A.~Boukhatem, and K.~El~Mufti.
\newblock Almost periodicity and stability for solutions to networks of beams
  with structural damping.
\newblock {\em Discrete Contin. Dyn. Syst. Ser. S}, 16:1201--1215, 2023.

\bibitem{MR4382292}
A.~Bchatnia, K.~El~Mufti, and R.~Yahia.
\newblock Stability of an infinite star-shaped network of strings by a
  {K}elvin-{V}oigt damping.
\newblock {\em Math. Methods Appl. Sci.}, 45:2024--2041, 2022.

\bibitem{Beck1952}
M.~Beck.
\newblock Die {K}nicklast des einseitig eingespannten, tangential gedrückten
  {S}tabes.
\newblock {\em Z. Angew. Math. Phys.}, 3:225--228, 1952.
\newblock (In German).

\bibitem{MR4433342}
G.~Berkolaiko and M.~Ettehad.
\newblock Three-dimensional elastic beam frames: rigid joint conditions in
  variational and differential formulation.
\newblock {\em Stud. Appl. Math.}, 148:1586--1623, 2022.

\bibitem{Bolotin1963}
V.~V. Bolotin.
\newblock {\em Nonconservative {P}roblems of the {T}heory of {E}lastic
  {S}tability}.
\newblock Pergamon, 1963.

\bibitem{MR2070601}
A.~V. Borovskikh and K.~P. Lazarev.
\newblock Fourth-order differential equations on geometric graphs.
\newblock {\em J. Math. Sci.}, 119:719--738, 2004.

\bibitem{MR4509852}
Z.~Bouallagui and M.~Jellouli.
\newblock Exponential stability of {R}ayleigh beam equation on a star-shaped
  network with indefinite damping.
\newblock {\em Math. Methods Appl. Sci.}, 45:10828--10851, 2022.

\bibitem{MR4504323}
Y.~Cheng, Y.~Wu, and B.~Z. Guo.
\newblock Boundary stability criterion for a nonlinear axially moving beam.
\newblock {\em IEEE Trans. Automat. Control}, 67:5714--5729, 2022.

\bibitem{CrandallEtAl1968}
S.~H. Crandall, D.~C. Karnopp, E.~F. Kurtz, and K.~Pridmore-Brown.
\newblock {\em Dynamics of {M}echanical and {E}lectromechanical {S}ystems}.
\newblock McGraw Hill, 1968.

\bibitem{CurtainZwart1995}
R.~F. Curtain and H.~J. Zwart.
\newblock {\em An {I}ntroduction to {I}nfinite-{D}imensional {L}inear {S}ystems
  {T}heory}.
\newblock Springer, 1995.

\bibitem{DagerZuazua2006}
R.~D\'ager and E.~Zuazua.
\newblock {\em Wave {P}rogagation, {O}bservation and {C}ontrol in $1–d$
  {F}lexible {M}ulti-{S}tructures}.
\newblock Springer, 2006.

\bibitem{MR4356898}
M.~Deliyianni and J.~T. Webster.
\newblock Theory of solutions for an inextensible cantilever.
\newblock {\em Appl. Math. Optim.}, 84:S1345--S1399, 2021.

\bibitem{DunfordSchwartz1971}
N.~Dunford and J.~T. Schwartz.
\newblock {\em Linear {O}perators.\ {P}art {III}: {S}pectral {O}perators}.
\newblock John Wiley \& Sons, 1971.

\bibitem{MR3655800}
P.~Freitas and J.~Lipovsk\'{y}.
\newblock Eigenvalue asymptotics for the damped wave equation on metric graphs.
\newblock {\em J. Differential Equations}, 263:2780--2811, 2017.

\bibitem{GohbergEtAl1990}
I.~Gohberg, S.~Goldberg, and M.~A. Kaashoek.
\newblock {\em Classes of {L}inear {O}perators.\ {V}ol.\ {I}}.
\newblock Birkh\"auser, 1990.

\bibitem{GohbergKrein1969}
I.~C. Gohberg and M.~G. Kre{\u\i}n.
\newblock {\em Introduction to the {T}heory of {L}inear {N}onselfadjoint
  {O}perators}.
\newblock American Mathematical Society, 1969.

\bibitem{MR4072654}
F.~Gregorio and D.~Mugnolo.
\newblock Bi-{L}aplacians on graphs and networks.
\newblock {\em J. Evol. Equ.}, 20:191--232, 2020.

\bibitem{GuoWang2019}
B.~Z. Guo and J.~M. Wang.
\newblock {\em Control of {W}ave and {B}eam {PDE}s: {T}he {R}iesz {B}asis
  {A}pproach}.
\newblock Springer, 2019.

\bibitem{MR2070993}
B.~Z. Guo and Y.~Xie.
\newblock Basis property and stabilization of a translating tensioned beam
  through a pointwise control force.
\newblock {\em Comput. Math. Appl.}, 47:1397--1409, 2004.

\bibitem{HAN1999935}
S.~M. Han, H.~Benaroya, and T.~Wei.
\newblock Dynamics of transversely vibrating beams using four engineering
  theories.
\newblock {\em J. Sound Vibration}, 225:935--988, 1999.

\bibitem{Kato1995}
T.~Kato.
\newblock {\em Perturbation {T}heory for {L}inear {O}perators}.
\newblock Springer, 1995.

\bibitem{MR4645077}
A.~Kelleche and F.~Saedpanah.
\newblock Stabilization of an axially moving {E}uler {B}ernoulli beam by an
  adaptive boundary control.
\newblock {\em J. Dyn. Control Syst.}, 29:1037--1054, 2023.

\bibitem{Khemmoudj2021}
A.~Khemmoudj.
\newblock Stabilisation of a viscoelastic beam conveying fluid.
\newblock {\em Internat.\ J.\ Control}, 94:235--247, 2021.

\bibitem{MR3349572}
J.~C. Kiik, P.~Kurasov, and M.~Usman.
\newblock On vertex conditions for elastic systems.
\newblock {\em Phys. Lett. A}, 379:1871--1876, 2015.

\bibitem{LagneseEtAl1994}
J.~E. Lagnese, G.~Leugering, and Schmidt E. J.~P. G.
\newblock {\em Modeling, {A}nalysis and {C}ontrol of {D}ynamic {E}lastic
  {M}ulti-{L}ink {S}tructures}.
\newblock Springer, 1994.

\bibitem{MR3739755}
Y.~F. Li, Z.~J. Han, and G.~Q. Xu.
\newblock Explicit decay rate for coupled string-beam system with localized
  frictional damping.
\newblock {\em Appl. Math. Lett.}, 78:51--58, 2018.

\bibitem{MahinzaeimEtAl2022}
M.~Mahinzaeim, G.~Q. Xu, and X.~X. Feng.
\newblock The {R}iesz basisness of the eigenfunctions and eigenvectors
  connected to the stability problem of a fluid-conveying tube with boundary
  control.
\newblock arXiv:2204.01432 [math.AP], April 2022.

\bibitem{MahinzaeimEtAl2021}
M.~Mahinzaeim, G.~Q. Xu, and H.~E. Zhang.
\newblock On the exponential stability of {B}eck's problem on a star-shaped
  graph.
\newblock {\em J. Differential Equations}, 374:410--445, 2023.

\bibitem{MalamudEtAl2012}
M.~M. Malamud and L.~L. Oridoroga.
\newblock On the completeness of root subspaces of boundary value problems for
  first order systems of ordinary differential equations.
\newblock {\em J. Funct. Anal.}, 263:1939--1980, 2012.

\bibitem{Miloslavskii1985}
A.~I. Miloslavskii.
\newblock Stability of certain classes of evolution equations.
\newblock {\em Sib. Math. J.}, 26:723--735, 1985.

\bibitem{Paidoussis2014}
M.~P. Pa\"idoussis.
\newblock {\em Fluid-{S}tructure {I}nteractions: {S}lender {S}tructures and
  {A}xial {F}low. {V}ol.\ {1}}.
\newblock Elsevier, 2014.

\bibitem{PAIDOUSSIS2022103664}
M.~P. Pa\"idoussis.
\newblock Pipes conveying fluid: a fertile dynamics problem.
\newblock {\em J.\ Fluids Struct.}, 114:103664, 2022.

\bibitem{Pazy1983}
A.~Pazy.
\newblock {\em Semigroups of {L}inear {O}perators and {A}pplications to
  {P}artial {D}ifferential {E}quations}.
\newblock Springer, 1983.

\bibitem{Roh1982}
H.~R{\"o}h.
\newblock Dissipative operators with finite dimensional damping.
\newblock {\em Proc.\ Roy.\ Soc.\ Edinburgh Sect.\ A}, 91:243--263, 1982.

\bibitem{10240878}
Y.~N. Wang and J.~M. Wang.
\newblock Boundary feedback stabilization of three coupled strings with joint
  anti-dampers.
\newblock In {\em 42nd {C}hinese {C}ontrol {C}onference}, pages 938--943, 2023.
\newblock Tianjin, China, July 24--26.

\bibitem{MR3908982}
Y.~R. Xie, Z.~J. Han, and G.~Q. Xu.
\newblock Stabilization of serially connected hybrid {PDE}-{ODE} system with
  unknown external disturbances.
\newblock {\em Appl. Anal.}, 98:718--734, 2019.

\bibitem{Xu2010}
G.~Q. Xu.
\newblock {\em Linear {O}perator {T}heory in {B}anach {S}paces}.
\newblock Xueyuan, 2011.
\newblock (In Chinese).

\bibitem{Xu2010graph}
G.~Q. Xu and N.~E. Mastorakis.
\newblock {\em Differential Equations on Metric Graph}.
\newblock WSEAS, 2010.

\bibitem{xu2005expansion}
G.~Q. Xu and S.~P. Yung.
\newblock The expansion of a semigroup and a {R}iesz basis criterion.
\newblock {\em J. Differential Equations}, 210:1--24, 2005.

\bibitem{MR4673490}
Y.~Zhang and L.~Liu.
\newblock Stabilization for weakly coupled string-riser system with partial
  frictional damping.
\newblock {\em Math. Methods Appl. Sci.}, 46:19429--19451, 2023.

\bibitem{MR4336448}
Y.~L. Zhang, M.~Zhu, D.~Li, and J.~M. Wang.
\newblock Stabilization of two coupled wave equations with joint anti-damping
  and non-collocated control.
\newblock {\em Automatica}, 135:109995, 2022.

\bibitem{MR4027612}
Y.~X. Zhang, Z.~J. Han, and G.~Q. Xu.
\newblock Stability and spectral properties of general tree-shaped wave
  networks with variable coefficients.
\newblock {\em Acta Appl. Math.}, 164:219--249, 2019.

\bibitem{Ziegler1977}
H.~Ziegler.
\newblock {\em Principles of {S}tructural {S}tability}.
\newblock Springer, 1977.

\end{thebibliography}

\end{document}